\documentclass[a4paper, cleveref, autoref, thm-restate]{lipics-v2021-adapted}
\usepackage{mathtools}
\usepackage{physics}
\usepackage{graphicx}
\usepackage[dvipsnames, table]{xcolor}
\usepackage{cite}
\usepackage{bm}
\usepackage{booktabs}
\usepackage{diagbox}
\usepackage{MnSymbol}
\usepackage{complexity}
\usepackage{subcaption}
\newsavebox{\imagebox}

\usepackage[capitalise, noabbrev]{cleveref}

\nolinenumbers
\hideLIPIcs

\definecolor{red30}{RGB}{229,242,251}
\definecolor{green30}{RGB}{132,195,237}
\definecolor{openColor}{RGB}{241,217,235}
\definecolor{openColorCond}{RGB}{207,128,187}
\definecolor{signalblue2}{RGB}{107,143,240}
\definecolor{signalblue3}{RGB}{146,177,240}
\definecolor{signalblue4}{RGB}{191,207,240}
\definecolor{maygreen3}{RGB}{221,233,197}
\definecolor{KITgray30}{RGB}{179,179,179}
\definecolor{KITlilac}{RGB}{160,0,120}
\definecolor{KITcyanblue}{RGB}{80,170,230}

\newcommand{\N}{\mathbb{N}}
\newcommand{\calG}{\mathcal{G}}
\newcommand{\calH}{\mathcal{H}}
\newcommand{\calF}{\mathcal{F}}

\newcommand{\calB}{\mathcal{B}}

\newcommand{\calM}{\mathcal{M}} 

\newcommand{\calQ}{\mathcal{Q}}
\newcommand{\starchrom}{\chi_{\mathrm{s}}}

\usepackage{xargs}
\usepackage{ifthen}

\newcommand{\lab}[1]{\text{\fontfamily{lmss}\selectfont #1}}

\newcommandx{\icn}[3]{%
   \ifthenelse{ \equal{#3}{} }
      {\ensuremath{\operatorname{ic}_{\text{\fontfamily{lmss}\selectfont #1}}^{#2}}}
      {\ensuremath{\operatorname{ic}_{\text{\fontfamily{lmss}\selectfont #1}}^{#2}(#3)}}
}
\newcommandx{\cn}[3]{%
   \ifthenelse{ \equal{#3}{} }
      {\ensuremath{\operatorname{c}_{\text{\fontfamily{lmss}\selectfont #1}}^{#2}}}
      {\ensuremath{\operatorname{c}_{\text{\fontfamily{lmss}\selectfont #1}}^{#2}(#3)}}
}

\newcommand{\setmid}{\ensuremath{\; \colon \;}}
\newcommand{\shortref}[1]{{\color{black}$\langle$\ref{#1}$\rangle$}}
\newcommand{\shortrefs}[2]{{\color{black}$\langle$\ref{#1},\ref{#2}$\rangle$}}

\newcommand{\ra}[1]{\renewcommand{\arraystretch}{#1}}
\newcolumntype{Y}{@{\extracolsep{3pt}}>{\centering\arraybackslash}X@{\extracolsep{0pt}}}

\DeclarePairedDelimiterX\set[1]\lbrace\rbrace{\def\given{\setmid}#1}

\DeclareMathOperator{\tw}{tw}
\DeclareMathOperator{\mad}{mad}

\DeclareMathOperator{\shift}{Sh}
\DeclareMathOperator{\Forb}{Forb}

\author{Miriam Goetze}{Karlsruhe Institute of Technology, Germany}{miriam.goetze@kit.edu}{https://orcid.org/0000-0001-8746-522X}{funded by the Deutsche Forschungsgemeinschaft (DFG, German Research Foundation) -- 520723789}

\author{Yidi Zang}{Karlsruhe Institute of Technology, Germany}{}{https://orcid.org/0009-0002-8024-7347}{}

\authorrunning{M. Goetze, Y. Zang}
\title{Boundedness and Separation Between Induced and Non-Induced Covering Numbers}

\newtheorem{hypothesis}[theorem]{Hypothesis}
\crefname{question}{Question}{Questions}
\Crefname{question}{Question}{Questions}
\crefname{conjecture}{Conjecture}{Conjectures}
\Crefname{conjecture}{Conjecture}{Conjectures}
\crefname{hypothesis}{Hypothesis}{Hypotheses}



\AIdecl{Generative AI has been used to obtain an initial overview of results in the literature implied or generalized by this work. The results have subsequently been verified by the authors.}

\begin{document}

\maketitle

\begin{abstract}
    There are four covering numbers~$\cn{g}{\calG}{H}, \cn{u}{\calG}{H}, \cn{l}{\calG}{H},\cn{f}{\calG}{H} $, each of which measures in a slightly different way how well the edges of a graph~$H$ (called a \emph{host}) can be covered with graphs of a class~$\calG$ (called a \emph{guest class}).
    If we require the graphs of~$\calG$ to correspond to induced subgraphs of~$H$, we obtain an induced variant~$\icn{x}{\calG}{}$ for each covering number~$\cn{x}{\calG}{}$ which satisfies $\cn{x}{\calG}{H} \leq \icn{x}{\calG}{H}$ for every graph~$H$.
    Yet, in general $\icn{x}{\calG}{}$ cannot be bounded in terms of~$\cn{x}{\calG}{}$.
    If there exists for a guest class~$\calG$ and a host class~$\calH$ a function~$f$ such that $\icn{x}{\calG}{H} \leq f(\cn{x}{\calG}{H})$ for every graph~$H \in \calH$, we call~$f$ a \emph{binding function}.
    
    Within this work, we study for which structural properties of a guest class~$\calG$ and a host class~$\calH$ such binding functions exist.
    We consider guest classes~$\calG$ that are monotone, hereditary, component-closed or neither, and have bounded maximum average degree, bounded chromatic number or neither.
    The host classes~$\calH$ we consider have bounded treewidth, exclude some minor, have bounded maximum average degree, bounded chromatic number, or none of these properties. 
    For $219$ out of the $240$~possible $3$-tuples of properties for~$\calG$ and~$\calH$ and covering numbers we either provide a binding function or an example where no such function exists.
    In particular, we show that such binding functions always exist for hereditary guest classes~$\calG$ of bounded maximum average degree for three of the four covering numbers, but may not for the fourth kind.
\end{abstract}

\section{Introduction}
\label{sec:introduction}

One branch of graph theory studies how well a given graph~$H$ can be decomposed into smaller pieces of a graph class~$\calG$.  
For a (not necessarily proper) $k$-vertex-coloring of~$H$, we can interpret the color classes as such pieces.
This yields a decomposition of the vertex set of~$H$ into independent sets in the case of proper colorings, showing a relation between the well-known chromatic number~$\chi(H)$ and such decompositions.

Within this work, we consider decompositions (more generally covers) of the edges of~$H$. 
Depending on how we measure the complexity of a cover of a graph~$H$ (the \emph{host}) with graphs of a class~$\calG$ (the \emph{guest class}), we obtain four different covering numbers: the global, union, local and folded covering numbers denoted by~$\cn{g}{\calG}{H}$, $\cn{u}{\calG}{H}$, $\cn{l}{\calG}{H}$ and~$\cn{f}{\calG}{H}$ respectively (cf. \cref{subparagraph:covering_numbers} for a formal definition). 
The covering numbers have been introduced by Knauer and Ueckerdt to unify the analysis and relationship of existing graph parameters \cite{Knauer2016_3w3c1g}.
Indeed, many parameters based on vertex- and edge-colorings correspond to a covering number for a suitable guest class~$\calG$.
Chromatic index (cover with matchings \cite{caoGraphEdgeColoring2019}), star arboricity (cover with stars \cite{goncalvesStarCaterpillarArboricity2009,Knauer2016_3w3c1g}), arboricity (cover with forests \cite{nash-william1964}) and thickness (cover with planar graphs \cite{mutzelThicknessGraphsSurvey1998}) can each be interpreted as a covering number.
\begin{table}[ht]
    \ra{1.3}
    \def\smallDist{3}
    \def\largeDist{8}
    \small
    \centering
    \begin{subtable}[t]{\textwidth}
        \centering
        \begin{tabularx}{0.83\textwidth}{c c@{\extracolsep{0pt}} @{\extracolsep{\smallDist pt}}c@{\extracolsep{0cm}} @{\extracolsep{\largeDist pt}}c@{\extracolsep{0cm}} @{\extracolsep{\smallDist pt}}c@{\extracolsep{0cm}} @{\extracolsep{\smallDist pt}}c@{\extracolsep{0cm}} @{\extracolsep{\largeDist pt}}c@{\extracolsep{0cm}} @{\extracolsep{\smallDist pt}}c@{\extracolsep{0cm}} @{\extracolsep{\smallDist pt}}c@{\extracolsep{0cm}}}
            \toprule
            \multirow{2}{*}{\diagbox[width=6em]{\phantom{vu}$\calH$}{$\calG$}} & \multirow{2}{*}{any} & \multirow{2}{*}{cc} & \multicolumn{3}{c}{hereditary} & \multicolumn{3}{c}{monotone} \\
            \arrayrulecolor{black}\cmidrule(r){4-6} \cmidrule(l){7-9}
            &   &    & \multicolumn{1}{c}{any} & $\chi$ & $\mad$ & \multicolumn{1}{c}{any} & $\chi$ & \multicolumn{1}{c}{$\mad$} \\
            \cmidrule{2-9}
            \arrayrulecolor{white}
            any & \multicolumn{1}{|c|}{\cellcolor{red30}\phantom{\shortref{lem:bounded_g_H_minor_G_hereditary}}} & \multicolumn{1}{|c|}{\cellcolor{red30}\phantom{\shortref{lem:bounded_g_H_minor_G_hereditary}}} & \multicolumn{1}{|c|}{\cellcolor{red30}\shortref{lem:sep_gulf_H_any_G_hereditary}} & \multicolumn{1}{|c|}{\cellcolor{red30}\phantom{\shortref{lem:bounded_g_H_minor_G_hereditary}}} & \multicolumn{1}{|c|}{\cellcolor{red30}\phantom{\shortref{prop:bounded_H_mad_tw_G_hereditary}}} & \multicolumn{1}{|c|}{\cellcolor{red30}\phantom{\shortref{prop:bounded_H_mad_tw_G_hereditary}}} & \multicolumn{1}{|c|}{\cellcolor{red30}\phantom{\shortref{prop:bounded_H_mad_tw_G_hereditary}}} & \multicolumn{1}{|c|}{\cellcolor{red30}\phantom{\shortref{lem:sep_g_H_mad_G_monotone}}} \\
            \cmidrule{2-9}
            bounded $\chi$ & \multicolumn{1}{|c|}{\cellcolor{red30}} & \multicolumn{1}{|c|}{\cellcolor{red30}} & \multicolumn{1}{|c|}{\cellcolor{red30}\phantom{\shortref{prop:bounded_H_mad_tw_G_hereditary}}} & \multicolumn{1}{|c|}{\cellcolor{red30}} & \multicolumn{1}{|c}{\cellcolor{red30}} & \multicolumn{1}{|c|}{\cellcolor{red30}} & \multicolumn{1}{|c|}{\cellcolor{red30}} & \multicolumn{1}{|c|}{\cellcolor{red30}}\\
             \cmidrule{2-9}
            bounded $\mad$ & \multicolumn{1}{|c|}{\cellcolor{red30}} & \multicolumn{1}{|c}{\cellcolor{red30}} & \multicolumn{1}{|c|}{\cellcolor{red30}} & \multicolumn{1}{|c|}{\cellcolor{red30}} & \multicolumn{1}{|c}{\cellcolor{red30}} & \multicolumn{1}{|c|}{\cellcolor{red30}} & \multicolumn{1}{|c|}{\cellcolor{red30}} & \multicolumn{1}{|c|}{\cellcolor{red30}\shortref{lem:sep_g_H_mad_G_monotone}}\\
             \cmidrule{2-9}
            $M$-minor-free & \multicolumn{1}{|c|}{\cellcolor{red30}} & \multicolumn{1}{|c}{\cellcolor{red30}\phantom{\shortref{lem:sep_gul_H_bounded_tw_G_cc}}} & \multicolumn{1}{|c|}{\cellcolor{green30}\shortref{lem:bounded_g_H_minor_G_hereditary}} & \multicolumn{1}{|c|}{\cellcolor{green30}} & \multicolumn{1}{|c}{\cellcolor{green30}} & \multicolumn{1}{|c|}{\cellcolor{green30}} & \multicolumn{1}{|c|}{\cellcolor{green30}} & \multicolumn{1}{|c|}{\cellcolor{green30}}\\
             \cmidrule{2-9}
            bounded $\tw$ & \multicolumn{1}{|c|}{\cellcolor{red30}\phantom{\shortref{lem:sep_gul_H_bounded_tw_G_cc}}} & \multicolumn{1}{|c|}{\cellcolor{red30}\shortref{lem:sep_gul_H_bounded_tw_G_cc}}& \multicolumn{1}{|c|}{\cellcolor{green30}\shortref{prop:bounded_H_mad_tw_G_hereditary}} & \multicolumn{1}{|c|}{\cellcolor{green30}} & \multicolumn{1}{|c|}{\cellcolor{green30}} & \multicolumn{1}{|c|}{\cellcolor{green30}\phantom{\shortref{prop:bounded_H_mad_tw_G_hereditary}}} & \multicolumn{1}{|c|}{\cellcolor{green30}} & \multicolumn{1}{|c|}{\cellcolor{green30}}\\
            \arrayrulecolor{black}\bottomrule
        \end{tabularx}
        
        \medskip
        
        \caption{Is $\calH$ $(\icn{g}{\calG}{},\cn{g}{\calG}{})$-bounded? \;\; \textcolor{green30}{$\blacksquare$} Always Yes \;\; \textcolor{red30}{$\blacksquare$} Sometimes No}
        \label{fig:global}
    \end{subtable}
    \begin{subtable}[t]{\textwidth}
        \centering
        \begin{tabularx}{0.83\textwidth}{c c@{\extracolsep{0pt}} @{\extracolsep{\smallDist pt}}c@{\extracolsep{0cm}} @{\extracolsep{\largeDist pt}}c@{\extracolsep{0cm}} @{\extracolsep{\smallDist pt}}c@{\extracolsep{0cm}} @{\extracolsep{\smallDist pt}}c@{\extracolsep{0cm}} @{\extracolsep{\largeDist pt}}c@{\extracolsep{0cm}} @{\extracolsep{\smallDist pt}}c@{\extracolsep{0cm}} @{\extracolsep{\smallDist pt}}c@{\extracolsep{0cm}}}
            \toprule
            \multirow{2}{*}{\diagbox[width=6em]{\phantom{vu}$\calH$}{$\calG$}} & \multirow{2}{*}{any} & \multirow{2}{*}{cc} & \multicolumn{3}{c}{hereditary} & \multicolumn{3}{c}{monotone} \\
            \arrayrulecolor{black}\cmidrule(r){4-6} \cmidrule(l){7-9}
            &   &    & \multicolumn{1}{c}{any} & $\chi$ & $\mad$ & \multicolumn{1}{c}{any} & $\chi$ & \multicolumn{1}{c}{$\mad$} \\
            \cmidrule{2-9}
            \arrayrulecolor{white}
            any & \multicolumn{1}{|c|}{\cellcolor{red30}\phantom{\shortref{lem:bounded_g_H_minor_G_hereditary}}} & \multicolumn{1}{|c|}{\cellcolor{red30}\phantom{\shortref{lem:bounded_g_H_minor_G_hereditary}}} & \multicolumn{1}{|c|}{\cellcolor{red30}\shortref{lem:sep_gulf_H_any_G_hereditary}} & \multicolumn{1}{|c|}{\cellcolor{openColor}\phantom{\shortref{lem:bounded_g_H_minor_G_hereditary}}} & \multicolumn{1}{|c|}{\cellcolor{green30}\shortref{prop:bounded_H_any_G_hereditary_mad}} & \multicolumn{1}{|c|}{\cellcolor{openColor}\phantom{\shortref{prop:bounded_H_mad_tw_G_hereditary}}} & \multicolumn{1}{|c|}{\cellcolor{openColor}\phantom{\shortref{prop:bounded_H_mad_tw_G_hereditary}}} & \multicolumn{1}{|c|}{\cellcolor{green30}\phantom{\shortref{prop:bounded_H_mad_tw_G_hereditary}}} \\
            \cmidrule{2-9}
            bounded $\chi$ & \multicolumn{1}{|c|}{\cellcolor{red30}} & \multicolumn{1}{|c|}{\cellcolor{red30}} & \multicolumn{1}{|c|}{\cellcolor{openColorCond}\shortref{cond_thm:bounded_c_gulf_H_chi_G_hereditary}} & \multicolumn{1}{|c|}{\cellcolor{openColorCond}} & \multicolumn{1}{|c}{\cellcolor{green30}} & \multicolumn{1}{|c|}{\cellcolor{openColorCond}} & \multicolumn{1}{|c|}{\cellcolor{openColorCond}} & \multicolumn{1}{|c|}{\cellcolor{green30}}\\
             \cmidrule{2-9}
            bounded $\mad$ & \multicolumn{1}{|c|}{\cellcolor{red30}} & \multicolumn{1}{|c}{\cellcolor{red30}} & \multicolumn{1}{|c|}{\cellcolor{green30}\shortref{prop:bounded_H_mad_tw_G_hereditary}} & \multicolumn{1}{|c|}{\cellcolor{green30}} & \multicolumn{1}{|c}{\cellcolor{green30}} & \multicolumn{1}{|c|}{\cellcolor{green30}\phantom{\shortref{prop:bounded_H_mad_tw_G_hereditary}}} & \multicolumn{1}{|c|}{\cellcolor{green30}} & \multicolumn{1}{|c|}{\cellcolor{green30}}\\
             \cmidrule{2-9}
            $M$-minor-free & \multicolumn{1}{|c|}{\cellcolor{red30}} & \multicolumn{1}{|c}{\cellcolor{red30}\phantom{\shortref{lem:sep_gul_H_bounded_tw_G_cc}}} & \multicolumn{1}{|c|}{\cellcolor{green30}} & \multicolumn{1}{|c|}{\cellcolor{green30}} & \multicolumn{1}{|c}{\cellcolor{green30}} & \multicolumn{1}{|c|}{\cellcolor{green30}} & \multicolumn{1}{|c|}{\cellcolor{green30}} & \multicolumn{1}{|c|}{\cellcolor{green30}}\\
             \cmidrule{2-9}
            bounded $\tw$ & \multicolumn{1}{|c|}{\cellcolor{red30}\phantom{\shortref{lem:sep_gul_H_bounded_tw_G_cc}}} & \multicolumn{1}{|c|}{\cellcolor{red30}\shortrefs{lem:sep_gul_H_bounded_tw_G_cc}{lem:sep_gulf_H_tw_2_G_cc}}& \multicolumn{1}{|c|}{\cellcolor{green30}} & \multicolumn{1}{|c|}{\cellcolor{green30}} & \multicolumn{1}{|c|}{\cellcolor{green30}} & \multicolumn{1}{|c|}{\cellcolor{green30}} & \multicolumn{1}{|c|}{\cellcolor{green30}} & \multicolumn{1}{|c|}{\cellcolor{green30}}\\
            \arrayrulecolor{black}\bottomrule
        \end{tabularx}
        
        \medskip
        
        \caption{Is $\calH$ $(\icn{x}{\calG}{},\cn{x}{\calG}{})$-bounded for $\lab{x} \in \set{\lab{u},\lab{l},\lab{f}}$? \;\; 
        $\begin{aligned}[t]
        &\text{\textcolor{green30}{$\blacksquare$} Always Yes \;\; \textcolor{red30}{$\blacksquare$} Sometimes No \;\; \textcolor{openColor}{$\blacksquare$} Unknown} \\
        &\text{\textcolor{openColorCond}{$\blacksquare$} Always yes assuming \cref{conjecture:bounded_u_H_bip_G_hereditary_bip}} 
        \end{aligned}$}
        \label{fig:union-local-folded}
    \end{subtable}
    \caption{Overview of the results in \cref{thm:global-introduction,thm:local-folded-introduction}. The two tables show whether any host class~$\calH$ with a sparsity restriction (having none (any), having bounded chromatic number ($\chi$) or bounded maximum average degree ($\mad$), being $M$-minor free for some~$M$, or having bounded treewidth ($\tw$)) is $(\icn{x}{\calG}{},\cn{x}{\calG}{})$-bounded for every guest class~$\calG$ with a sparsity restriction and a closure restriction (having none (any), being component-closed (cc), being hereditary or monotone). Numbers~$\langle X \rangle$ refer to the corresponding Theorem~$X$ in the paper, while all unnumbered results are immediate consequences.}
    \label{tab:overview}
\end{table}

Yet, in some applications we only consider covers where each piece corresponds to an \emph{induced} subgraph of~$H$. 
This is for example the case for strong edge colorings\footnote{also called \emph{distance-$2$-edge-colorings}} (cover with induced matchings) \cite{faudree1989inducedmatchings}, as well as induced and weak induced arboricity (cover with induced forests) \cite{axenovich2019induced}.
We thus extend the covering number framework by introducing four induced covering numbers~$\icn{g}{\calG}{H}$, $\icn{u}{\calG}{H}$, $\icn{l}{\calG}{H}$ and~$\icn{f}{\calG}{H}$ (cf. \cref{subparagraph:induced_covering_numbers} for a formal definition). 
Each induced covering number is more restrictive than its non-induced counterpart.
\begin{observation}
    For every $\lab{x} \in \set{\lab{g}, \lab{u}, \lab{l}, \lab{f}}$, every graph class~$\calG$ and every graph~$H$, we have 
    \[\cn{x}{\calG}{H} \leq \icn{x}{\calG}{H}.\]
\end{observation}
In general, the induced covering number~$\icn{x}{\calG}{}$ cannot be bounded in terms of~$\cn{x}{\calG}{}$ for a host class~$\calH$.
Consider for example the graph class~$\calH$ consisting of all complete graphs and the class~$\calG$ consisting of~$K_2$ and graphs obtained from complete graphs by deleting one edge.
While we can cover all edges of~$K_n$ with two subgraphs~$K_n-K_2$ and $K_2$ in $\calG$, the only graph of~$\calG$ that is also an induced subgraph of~$K_n$ is the graph~$K_2$ on a single edge. 
Hence, the minimum number~$\cn{g}{\calG}{K_n}$ of graphs of~$\calG$ that cover all edges of~$K_n$ is~$2$ for every $n \geq 3$;
yet, if we require the subgraphs to be induced, the minimum number~$\icn{g}{\calG}{K_n}$ of such subgraphs is $\abs{E(K_n)}$.

\subparagraph*{Our results.}
Within this work we give restrictions on~$\calG$ and $\calH$ under which there exists a function~$f\colon \N \to \N$ such that for every graph~$H \in \calH$, we have $\icn{x}{\calG}{H} \leq f(\cn{x}{\calG}{H})$. 
In that case, we say that~$\calH$ is \emph{$(\icn{x}{\calG}{},\cn{x}{\calG}{})$-bounded}.
Otherwise, we say that $\icn{x}{\calG}{}$ and $\cn{x}{\calG}{}$ can be \emph{separated} for the graph class~$\calH$.
We consider two kinds of restrictions, the same as have been considered in \cite{goetze2025boundednessseparationgraphcovering} with regard to $(\cn{u}{\calG}{},\cn{l}{\calG}{})$-, and $(\cn{l}{\calG}{},\cn{f}{\calG}{})$-boundedness:
\begin{enumerate}
\item \emph{Sparsity restrictions} limit the edge density of a graph class~$\calQ$. 
We consider graph classes that have bounded \emph{treewidth}, are \emph{$M$-minor-free} for some graph~$M$, have bounded \emph{maximum average degree}, have bounded chromatic number or no such restriction. 
\item \emph{Closure restrictions} ensure that certain subgraphs of each~$G \in \calQ$ also lie in the class~$\calQ$. We consider graph classes that are \emph{monotone}, \emph{hereditary}, \emph{component-closed} or have no such restriction. 
\end{enumerate}
See \cref{subparagraph:parameters} for definitions.
Both the sparsity and the closure restrictions are less and less restrictive in the order given above. 
That is, any graph class that has bounded maximum average degree has also bounded chromatic number (but not vice versa), and every monotone graph class is also component-closed.
The aim of this work is to determine the least restrictive properties of graph classes~$\calG$ and~$\calH$ such that~$\calH$ is $(\icn{x}{\calG}{},\cn{x}{\calG}{})$-bounded (independent of the choice of~$\calG$ and~$\calH$).
See \cref{tab:overview} for an overview of our results.

\begin{theorem}[global]
    \label{thm:global-introduction}{\ \\}
    For any two graph classes $\calG,\calH$ we have that $\calH$ is $(\icn{g}{\calG}{},\cn{g}{\calG}{})$-bounded, provided the following holds:
    \begin{enumerate}
        \item $\calG$ is hereditary and $\calH$ is $M$-minor-free for some graph~$M$.
    \end{enumerate}
    On the other hand, there exist two graph classes $\calG,\calH$ such that $\calH$ is \emph{not} $(\icn{g}{\calG}{},\cn{g}{\calG}{})$-bounded in each of the following cases:
    \begin{enumerate}
        \item $\calG$ is monotone and both~$\calG$ and~$\calH$ have bounded maximum average degree.
        \item $\calG$ is component-closed and $\calH$ has bounded treewidth.
    \end{enumerate}    
\end{theorem}

\begin{theorem}[union, local, folded]
    \label{thm:local-folded-introduction}{\ \\}
    For any two graph classes $\calG,\calH$ and $\lab{x} \in \set{\lab{u},\lab{l},\lab{f}}$ we have that $\calH$ is $(\icn{x}{\calG}{},\cn{x}{\calG}{})$-bounded, provided one of the following holds:
    \begin{enumerate}
        \item $\calG$ is hereditary and has bounded maximum average degree.
        \item $\calG$ is hereditary and~$\calH$ has bounded maximum average degree.
    \end{enumerate}
    On the other hand, there exist two graph classes $\calG,\calH$ such that $\calH$ is \emph{not} $(\icn{x}{\calG}{},\cn{x}{\calG}{})$-bounded in each of the following cases:
    \begin{enumerate}
        \item $\calG$ is hereditary.
        \item $\calG$ is component-closed and $\calH$ has bounded treewidth.\label{item:local-folded-4}
    \end{enumerate}    
\end{theorem}

Let us briefly discuss two consequences of our main results.
Some of our results generalize previous known bounds.
Axenovich, Dörr, Rollin and Ueckerdt bound the \emph{weak induced arboricity} (that is the induced union $\calF$-covering number~$\icn{u}{\calF}{}$ for the class $\calF$ of forests) in terms of the arboricity~$\cn{u}{\calF}{}$. 
They show that $\icn{u}{\calF}{H} \leq 4(\cn{u}{\calF}{H})^2$ for every graph~$H$ \cite[Claim~2]{axenovich2019induced}. 
In fact, \cref{prop:bounded_H_any_G_hereditary_mad} provides a generalization thereof for hereditary guest classes of bounded maximum average degree, which yields the same bound as in \cite[Claim~2]{axenovich2019induced} when covering with forests.

When the induced $\calG$-covering number~$\icn{x}{\calG}{}$ is bounded by a linear function in~$\cn{x}{\calG}{}$ and further~$\cn{x}{\calG}{}$ can be determined in polynomial time, we obtain a constant factor approximation of~$\icn{x}{\calG}{}$.
This is for example the case for covers of planar graphs with matchings. 
For the class~$\calM$ of matchings, the induced covering number~$\icn{g}{\calM}{H}$ of a graph~$H$ corresponds to the \emph{strong chromatic index}, the covering number~$\cn{g}{\calM}{H}$ to the chromatic index.
While~$\cn{g}{\calM}{H}$ can be determined in polynomial time when~$H$ is bipartite \cite{schrijver1998BipEdgeCol}, deciding whether~$\icn{g}{\calM}{G} \leq 6$ is $\NP$-hard even if~$H$ is a planar, bipartite graph of maximum degree~$3$ \cite{hocquardStrongEdgecolouringInduced2013}. 
In fact, \cref{lem:bounded_g_H_minor_G_hereditary} shows that $\cn{g}{\calM}{H}$ yields an $800$-approximation of~$\icn{g}{\calM}{H}$ for every planar, bipartite host graph~$H$ as planar graphs contain no~$K_5$ as a minor.
Yet, note that this merely shows how boundedness can be used for deriving approximation algorithms.
The approach does not provide the best constants.
Indeed, in the case of planar, bipartite hosts, the approximation algorithm can easily be improved to a $4$-approximation (using a slight adaptation of\cite[Theorem~10]{faudree1990strong}): a bipartite graph of maximum degree~$\Delta$ admits a proper $\Delta$-edge-coloring. 
Each color class can be decomposed into up to four induced matchings using a polynomial time algorithm for computing a $4$-vertex-coloring of a planar graph \cite{robertson1996efficiently}.
Similarly, we obtain an $11$-approximation of~$\icn{g}{\calM}{H}$ when~$H$ is planar, bipartite using $\icn{g}{\calM}{H} \leq (2 \mad(H)-1) \cdot \cn{g}{\calM}{H}$ (the bound also holds when~$H$ is not planar) \cite[Theorem~13]{wangStrongChromaticIndex2025}, as planar graphs have maximum average degree less than~$6$. 

\section{Preliminaries}

\subparagraph*{Covering numbers.}
\label{subparagraph:covering_numbers}

A $\calG$-cover of a graph~$H$ is an edge-surjective graph homomorphism~$\varphi\colon G_1 \cupdot \dots \cupdot G_t \to H$ where~$G_i \in \calG$ for all $i \in [t]$. 
We call~$H$ the \emph{host} and the graphs~$G_i$ \emph{guests}.
The cover~$\varphi$ is \emph{injective} if each guest~$G_i$ is isomorphic to its image~$\varphi(G_i) \subseteq H$, that is the subgraph~$\varphi(G_i) \subseteq H$ is a copy of~$G_i$.
A cover $\varphi\colon G_1 \cupdot \dots \cupdot G_t \to H$ is\footnote{This definition deviates from the definitions used in \cite{Knauer2016_3w3c1g}. Knauer and Ueckerdt only define $k$-global, and $k$-local without requiring the cover to be injective. Yet, the resulting covering numbers are the same, as they also only consider injective covers in the case of global and local covering numbers.}
\begin{itemize}
    \item \emph{$k$-global} if~$\varphi$ is injective and $t \leq k$,
    \item \emph{$k$-local} if~$\varphi$ is injective and $\abs{\varphi^{-1}(v)} \leq k$ for all $v \in V(H)$,
    \item \emph{$k$-folded} if $\abs{\varphi^{-1}(v)} \leq k$ for all $v \in V(H)$.
\end{itemize}
That is, a $k$-global $\calG$-cover of a graph~$H$ consists of up to $k$ subgraphs~$G_1, \dots, G_k \in \calG$ of~$H$ whose union covers all edges of~$H$, see \cref{fig:examples_left} for an example.
In a $k$-local cover, the edges of~$H$ may be covered with any number of subgraphs~$G_1, \dots, G_t \in \calG$, yet each vertex is contained in at most~$k$ such subgraphs, see \cref{fig:examples_middle}.
However, while each vertex of~$H$ is contained in at most~$k$ of the graphs~$G_1, \dots, G_t$ of a $k$-folded $\calG$-cover, the graphs~$G_1, \dots, G_t$ need not be subgraphs of~$H$, see \cref{fig:examples_right}.
It is only required that their homomorphic images $\varphi(G_i)$ are subgraphs of $H$.
Note in particular that every $k$-global $\calG$-cover is also $k$-local, and every $k$-local cover is also $k$-folded.
\begin{figure}
    \centering
    \begin{subfigure}[t]{0.37\textwidth}
       \centering
        \includegraphics[page=5]{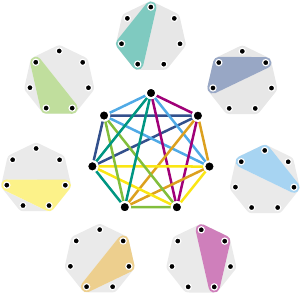}
        \caption{}
        \label{fig:examples_left}
    \end{subfigure}\hfill
    \begin{subfigure}[t]{0.37\textwidth}
        \centering
        \includegraphics[page=6]{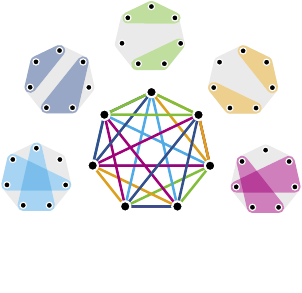}
        \caption{}
        \label{fig:examples_middle}
    \end{subfigure}\hfill
    \begin{subfigure}[t]{0.21\textwidth}
       \centering
        \includegraphics[page=7]{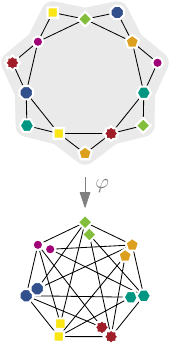}
        \caption{}
        \label{fig:examples_right}
    \end{subfigure}
    \caption{Examples of $\calG$-covers taken from \cite[Figure~1]{goetze2025boundednessseparationgraphcovering}. (\subref{fig:examples_left}) A $7$-global $3$-local $\set{K_3}$-cover of~$K_7$ certifying~$\cn{g}{\set{K_3}}{K_7} \leq 7$ and $\cn{l}{\set{K_3}}{K_7} \leq 3$. (\subref{fig:examples_middle}) A $5$-global $5$-local $\overline{\set{K_3}}$-cover of~$K_7$ certifying $\cn{u}{\set{K_3}}{K_7} \leq 5$. (\subref{fig:examples_right}) A $2$-folded $\Forb(C_4)$-cover of~$K_7$ certifying $\cn{f}{\Forb(C_4)}{K_7} \leq 2$ for the class~$\Forb(C_4)$ consisting of all $C_4$-free graphs.}
    \label{fig:example_non-induced}
\end{figure}

Covering numbers describe how well a graph can be decomposed into smaller pieces of a graph class~$\calG$, or the union-closure~$\overline{\calG}$ of~$\calG$ (the graph class consisting of all vertex-disjoint unions of graphs in~$\calG$).
We are interested in covers with as few pieces as possible.
Depending on how we count the number of pieces, we obtain four different covering numbers as have been defined by Knauer and Ueckerdt in \cite{Knauer2016_3w3c1g}\footnote{In \cite{Knauer2016_3w3c1g} Knauer and Ueckerdt only describe three covering numbers. These have been later on be complemented by the union-covering number \cite{blasius2018local_boxicity,merker2019local_page_numbers}. Our definition of the folded covering number~$\cn{f}{\calG}{}$ differs slightly from the one given in \cite{Knauer2016_3w3c1g}. They only consider folded covers with a single guest graph. Yet, for union-closed graph classes, the two definitions coincide.}:
\begin{itemize}
\item The \emph{global~$\calG$-covering number}~$\cn{g}{\calG}{H}$ of a graph~$H$ is the smallest~$t$ such that there exists (an injective) $t$-global $\calG$-cover $\varphi\colon G_1 \cupdot \dots \cupdot G_t \to H$ of~$H$.
\item The \emph{union~$\calG$-covering number}~$\cn{u}{\calG}{H}$ of a graph~$H$ is the smallest~$t$ such that there exists (an injective) $t$-global $\overline{\calG}$-cover $\varphi\colon G_1 \cupdot \dots \cupdot G_t \to H$ of~$H$.
\item The \emph{local~$\calG$-covering number}~$\cn{l}{\calG}{H}$ of a graph~$H$ is the smallest~$\ell$ such that there exists (an injective) $\ell$-local $\calG$-cover $\varphi\colon G_1 \cupdot \dots \cupdot G_t \to H$ of~$H$.
\item The \emph{folded~$\calG$-covering number}~$\cn{f}{\calG}{H}$ of a graph~$H$ is the smallest~$\ell$ such that there exists a (not necessarily injective) $\ell$-folded $\calG$-cover $\varphi\colon G_1 \cupdot \dots \cupdot G_t \to H$ of~$H$.
\end{itemize}
There is a hierarchy on the four covering numbers: for every graph class~$\calG$ and every graph~$H$ we have 
\[ \cn{f}{\calG}{H} \leq \cn{l}{\calG}{H} \leq \cn{u}{\calG}{H} \leq \cn{g}{\calG}{H}.\]
Yet, the gaps between two covering numbers can be arbitrarily large  \cite{goetze2025boundednessseparationgraphcovering}.

\subparagraph*{Induced covering numbers.}
\label{subparagraph:induced_covering_numbers}

A subgraph~$H'$ of a graph~$H$ is \emph{weak induced} if each component of~$H'$ is an induced subgraph of~$H$.
We say that a $\calG$-cover $\varphi\colon G_1 \cupdot \dots \cupdot G_t \to H$ is \emph{induced} (\emph{weak induced}) if each image~$\varphi(G_i)$ is an induced (weak induced) subgraph of~$H$.
Similarly to (non-induced) covering numbers, we call a cover~$\varphi\colon G_1 \cupdot \dots \cupdot G_t \to H$
\begin{itemize}
    \item \emph{$k$-global-induced} if~$\varphi$ is induced and $k$-global
    \item \emph{$k$-local-induced} if~$\varphi$ is weak induced and $k$-local
    \item \emph{$k$-folded-induced} if~$\varphi$ is weak induced and $k$-folded.
\end{itemize}
See \cref{fig:example-induced} for examples.
\begin{figure}
    \centering
    \begin{subfigure}[t]{0.3\textwidth}
       \centering
        \includegraphics[page=4]{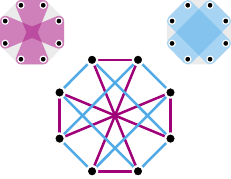}
        \caption{}
        \label{fig:examples_induced_left}
    \end{subfigure}\hfill
    \begin{subfigure}[t]{0.3\textwidth}
        \centering
        \includegraphics[page=5]{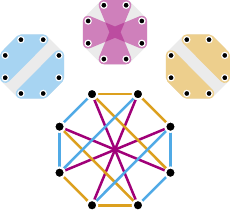}
        \caption{}
        \label{fig:examples_induced_middle}
    \end{subfigure}\hfill
    \begin{subfigure}[t]{0.37\textwidth}
       \centering
        \includegraphics[page=6]{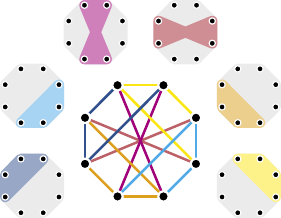}
        \caption{}
        \label{fig:examples_induced_right}
    \end{subfigure}
    \caption{Examples of non-induced, weak induced and induced $\calG$-covers of a graph~$H$. (\subref{fig:examples_induced_left}) A non-induced $2$-global $\overline{\set{C_4}}$-cover of~$H$ certifying~$\cn{l}{\set{C_4}}{H} \leq \cn{u}{\set{C_4}}{H} \leq 2$. (\subref{fig:examples_induced_middle}) A weak induced $3$-global $3$-local $\overline{\set{C_4}}$-cover of~$H$ certifying $\icn{l}{\set{C_4}}{H}\leq \icn{u}{\set{C_4}}{H} \leq 3$. (\subref{fig:examples_induced_right}) A $6$-global-induced $\set{C_4}$-cover of~$H$ certifying $\icn{g}{\set{C_4}}{H} \leq 6$.}
    \label{fig:example-induced}
\end{figure}
Considering (weak) induced $\calG$-covers yields a variant of each of the covering numbers above:
\begin{itemize}
\item The \emph{global-induced~$\calG$-covering number}~$\icn{g}{\calG}{H}$ of a graph~$H$ is the smallest~$t$ such that there exists an induced $t$-global $\calG$-cover $\varphi\colon G_1 \cupdot \dots \cupdot G_t \to H$ of~$H$.
\item The \emph{union-induced~$\calG$-covering number}~$\icn{u}{\calG}{H}$ of a graph~$H$ is the smallest~$t$ such that there exists a weak induced, $t$-global $\overline{\calG}$-cover $\varphi\colon G_1 \cupdot \dots \cupdot G_t \to H$ of~$H$.
\item The \emph{local-induced~$\calG$-covering number}~$\icn{l}{\calG}{H}$ of a graph~$H$ is the smallest~$\ell$ such that there exists a weak induced, $\ell$-local $\calG$-cover $\varphi\colon G_1 \cupdot \dots \cupdot G_t \to H$ of~$H$.
\item The \emph{folded-induced~$\calG$-covering number}~$\icn{f}{\calG}{H}$ of a graph~$H$ is the smallest~$\ell$ such that there exists a (not necessarily injective) weak induced, $\ell$-folded $\calG$-cover $\varphi\colon G_1 \cupdot \dots \cupdot G_t \to H$ of~$H$.
\end{itemize}
As for covering numbers, the covers we consider are less and less restrictive. 
That is, we obtain an order on the induced covering numbers:
\[ \icn{f}{\calG}{H} \leq \icn{l}{\calG}{H} \leq \icn{u}{\calG}{H} \leq \icn{g}{\calG}{H}.\]

\subparagraph*{Restrictions on guests and hosts.}
\label{subparagraph:parameters}
A graph class~$\calG$ is \emph{component-closed} if each component of any graph~$G \in \calG$ also lies in~$\calG$.
If all induced subgraphs of every graph~$G \in \calG$ lie in~$\calG$, the class is \emph{hereditary}.
The class~$\calG$ is called \emph{monotone} if~$\calG$ is closed under taking subgraphs.
As many graph parameters are based on covers with such monotone (for example in the case of thickness, that is covers with planar graphs \cite{mutzelThicknessGraphsSurvey1998}) or hereditary (in the case of edge clique number, that is covers with complete graphs \cite{brighamUpperBoundsEdge1984a}) graph classes, we study guest classes which fulfill these properties with respect to covering numbers.

In general, a graph~$H$ on $n$~vertices can have~$\Omega(n^2)$~edges and thus cannot be decomposed into few forests (each forest contains at most~$n-1$ edges).
Yet, if~$H$ is sparse, that is if~$H$ contains few edges with respect to~$n$, this is often possible, even if we require the forests to be (weak) induced \cite{axenovich2018kstronginducedarboricity} (cf. \cref{lem:mad_bounded_split_into_weak_induced_star_forests}).
For a graph parameter~$p$ we denote by~$p(\calH) = \sup_{H \in \calH} p(H)$ the maximum possible value of~$p(H)$ on~$\calH$. 
If~$p(\calH) < \infty$, we say that~$p$ is bounded on~$\calH$.
Within this work, we consider four different notions of sparsity of a graph class~$\calH$,
\begin{itemize}
    \item the \emph{chromatic number} $\chi(\calH)$ is bounded,
    \item the \emph{maximum average degree}~$\mad(\calH)$ is bounded,
    \item $\calH$ is \emph{$M$-minor-free} for some graph~$M$,
    \item the \emph{treewidth}~$\tw(\calH)$ is bounded.
\end{itemize}
In fact, each of these properties implies all of the properties above it.

The \emph{chromatic number}~$\chi(G)$ of a graph~$G$ is the smallest~$k \in \N$ for which~$G$ admits a $k$-vertex-coloring where no two adjacent vertices receive the same color.
The \emph{maximum average degree}~$\mad(G)$ of~$G$ is defined via the relation between the number of edges and vertices within its subgraphs:
\[
\mad(G) = \max \set{\frac{2\abs{E(G')}}{\abs{V(G')}} \given \text{$G'$ is a subgraph of~$G$ with~$\abs{V(G')} \geq 1$}}
\]
A graph~$G$ is \emph{$M$-minor-free} for some graph~$M$ if it does not contain~$M$ as a minor.
We define treewidth via \emph{$k$-trees}. 
For an integer~$k$, a clique on~$k+1$ vertices is a $k$-tree. 
The same holds for any graph~$H$ obtained from a $k$-tree~$H'$ by adding a vertex~$x$ and connecting~$x$ to all vertices of some $k$-clique in~$H'$. 
The minimum~$k$ such that a graph~$G$ is a subgraph of a $k$-tree~$H$ is the \emph{treewidth} of~$G$.

We consider the same structural restrictions as in \cite{goetze2025boundednessseparationgraphcovering}.
That is, the guest class may have no further restrictions, be component-closed, hereditary or monotone, as well as have no further restriction with respect to sparsity, have bounded chromatic number, or bounded maximum average degree.
For the host class~$\calH$, we only study sparsity restrictions: $\calH$ may have no restriction, have bounded chromatic number, bounded maximum average degree, be $M$-minor-free for some graph~$M$ or have bounded treewidth. 
See \cref{tab:overview} for an overview of the different results.
Note that boundedness-results in one setting may imply such results in other settings, and the same holds for separation. 
For example, if we show $(\icn{l}{\calG}{},\cn{l}{\calG}{})$-boundedness for every hereditary guest class~$\calG$ of bounded maximum average degree and every host class~$\calH$, we also obtain $(\icn{l}{\calG'}{},\cn{l}{\calG'}{})$-boundedness for monotone guest classes~$\calG'$ with $\mad(\calG') < \infty$ and any host class~$\calH'$ of bounded treewidth.

\section{General guest classes and hosts of bounded treewidth}
\label{sec:general_guests_tw}

If there are no further restrictions on the guest class~$\calG$, the induced and non-induced covering numbers can be arbitrarily far apart, even for very well-structured hosts. 
This corresponds to the results of the first two columns of \cref{fig:global,fig:union-local-folded}.
We first show that host classes of bounded treewidth may not be $(\icn{x}{\calG}{},\cn{x}{\calG}{})$-bounded for all~$\lab{x}\neq \lab{f}$ (cf. \cref{lem:sep_gul_H_bounded_tw_G_cc}), and then generalize this result to folded-induced covering numbers (cf. \cref{lem:sep_gulf_H_tw_2_G_cc}).

To show that an induced covering number~$\icn{x}{\calG}{}$ cannot be bounded in terms of the corresponding covering number~$\cn{x}{\calG}{}$, we generally choose the guest class~$\calG$ such that each host~$H \in \calH$ can be covered with few guests, yet~$H$ only contains small graphs of~$\calG$ as induced subgraphs.
In fact, there are only very few options to embed a (large) guest~$G \in \calG$ with a vertex of high degree in~$H$, if~$H$ only contains few such vertices.

\begin{observation}
\label{lem:degree}
    Let~$\varphi\colon G_1 \cupdot \dots \cupdot G_t \to H$ be a $\calG$-cover of a graph~$H$.
    \begin{enumerate}[(i)]
        \item\label{itm:degree_injective} If $\varphi$ is injective, then $\deg_{G_i}(u) \leq \deg_H(\varphi(u))$ for every vertex~$u \in V(G_i)$.
        \item\label{itm:degree_folded} If $\varphi$ is $\ell$-folded, then $\deg_{G_i}(u) \leq \ell \cdot \deg_H(\varphi(u))$
        for every vertex~$u \in V(G_i)$.
    \end{enumerate}
\end{observation}

The construction which shows that the global-, union-, and local-induced covering numbers cannot be bounded in terms of their non-induced counterparts evolves around graphs known as \emph{fans}.
Consider two copies of a path~$P_{k^k}$ on $k^k$~vertices and let~$x,y$ be one endpoint of the first and of the second copy respectively. 
We call the graph~$F_k'$ that is obtained from the two copies~$P_{k^k}$ by adding a universal vertex~$f$ (i.e., a vertex adjacent to every vertex of the two paths) a \emph{broken fan}; see \cref{fig:fans} for an illustration. 
The \emph{fan}~$F_k$ is the supergraph of~$F_k'$ where the edge~$xy$ is added.
\begin{figure}
    \centering
    \begin{subfigure}[t]{.4\linewidth}
    \centering
        \includegraphics[page=2]{figures/fans.pdf}
        \subcaption{}
        \label{fig:broken_fan}
    \end{subfigure}\hfil
    \begin{subfigure}[t]{.4\linewidth}
    \centering
        \includegraphics[page=1]{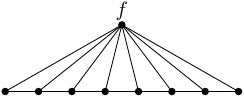}
        \subcaption{}
        \label{fig:fan}
    \end{subfigure}
    \caption{(\subref{fig:broken_fan}) A broken fan~$F_k'$  and (\subref{fig:fan}) a fan~$F_k$ for $k=2$.}
    \label{fig:fans}
\end{figure}
We call~$f$ (respectively~$f'$) the \emph{tip} of~$F_k$ (respectively $F_k'$) and the path~$P_{2k^k}$ (respectively $P_{k^k} \cupdot P_{k^k}$) its underlying \emph{(broken) path}.

\begin{proposition}
\label{lem:sep_gul_H_bounded_tw_G_cc}
    There is a host class~$\calH$ with $\tw(\calH) = 2$ and a component-closed guest class~$\calG$ such that for every $\lab{x} \in \set{\lab{g}, \lab{u},\lab{l}}$, the host class~$\calH$ is not $(\icn{x}{\calG}{},\cn{x}{\calG}{})$-bounded.
\end{proposition}
\begin{proof}
    The guest class~$\calG = \set{F_k' \given k \geq 2} \cup \set{K_2}$ consists of all broken fans and the host class~$\calH = \set{F_k \given k \geq 2}$ of all fans.
    As each~$G \in \calG$ is connected, the class~$\calG$ is component-closed.
    Note that~$\calH$ has treewidth~$2$.
    We need to show that for every $k \in \N$, we have
    \begin{enumerate}[(i)]
        \item\label{itm:wheel_cn} $\cn{l}{\calG}{F_k} \leq \cn{u}{\calG}{F_k} \leq \cn{g}{\calG}{F_k} \leq 2$
        \item\label{itm:wheel_icn} $\icn{g}{\calG}{F_k} \geq \icn{u}{\calG}{F_k} \geq \icn{l}{\calG}{F_k} \geq k$.
    \end{enumerate}
    It then follows that $\icn{x}{\calG}{}$ cannot be bounded in terms of $\cn{x}{\calG}{}$ for the host class~$\calH$ for any~$\lab{x} \neq \lab{f}$.

    To prove \eqref{itm:wheel_cn}, observe that $F_{k}'$ covers all but one edge (namely the edge~$xy$) of~$F_k$. 
    Thus, $\varphi\colon F_k' \cupdot K_2 \to F_k$ is a $2$-global $\calG$-cover certifying $\cn{g}{\calG}{F_k} \leq 2$.

    \smallskip

    To prove \eqref{itm:wheel_icn}, consider an injective $\ell$-local-induced $\calG$-cover $\varphi\colon G_1 \cupdot \dots \cupdot G_t \to F_k$.
    
    We show that~$\ell \geq k$ or each graph~$G_i$ covers at most $2k^{k-1}$ edges incident to the tip~$f$ of~$F_k$.
    Claim \eqref{itm:wheel_icn} then follows, as the tip~$f$ has to be hit at least~$\frac{2k^k}{2k^{k-1}} = k$ times to cover all edges incident to~$f$.
    
    The claim clearly holds for guests~$G_i$ which correspond to~$K_2$.
    We may thus assume that~$G_i$ is a broken fan~$F_q'$.
    As $\varphi$ is injective, $F_q'$ is a subgraph of~$F_k$. 
    In particular, we have~$q \leq k$. 
    Let~$f'$ denote the tip of~$F_q'$. 
    As every tip has degree at least~$8$, and all other vertices in a fan have degree at most~$3$, \cref{lem:degree}\,\eqref{itm:degree_injective} yields $\varphi(f') = f$. 
    Yet, if~$q < k$, we have~$\abs{V(F_q' - f')} \leq 2k^{k-1}$, i.e. $F_q'$ covers at most $2k^{k-1}$ edges incident to~$f$.
    If~$q=k$, the underlying broken path of~$F_ k'$ is mapped to the path of~$F_k$ as $\varphi(N(f')) \subseteq N(f)$.
    Thus, $\varphi(F_k')$ corresponds to a copy of~$F_k'$ in~$F_k$; a contradiction to $\varphi(F_k')$ being an induced subgraph. 
\end{proof}

However, \cref{lem:sep_gul_H_bounded_tw_G_cc} cannot be easily extended to folded covers. 
Indeed, every fan~$F_k$ admits a $3$-folded-induced cover with two copies of~$F_k'$ and a single edge~$K_2$, see \cref{fig:3-folded-induced-broken-fan-cover} for an illustration.
\begin{figure}
    \centering
    \includegraphics[page=3]{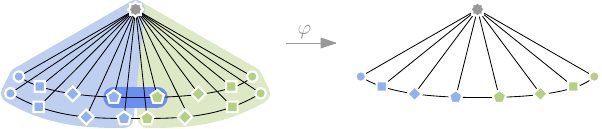}
    \caption{A $3$-folded-induced cover~$\varphi$ of~$F_k$ formed by two broken fans~$F_k'$ (\textcolor{signalblue4}{$\blacksquare$} and \textcolor{maygreen3}{$\blacksquare$}) and one~$K_2$ (\textcolor{signalblue2}{$\blacksquare$}) for $k=2$.}
    \label{fig:3-folded-induced-broken-fan-cover}
\end{figure}

While we do not have $(\icn{x}{\calG}{},\cn{x}{\calG}{})$-boundedness for all host classes of bounded treewidth, for forests (i.e. for graphs~$F$ with $\tw(F) \leq 1$) we do.

\begin{lemma}
    \label{lem:forests}
    For every graph class~$\calG$, each of the following holds.
    \begin{enumerate}[(i)]
        \item For every forest $F$ and $\lab{x} \in \set{\lab{u}, \lab{l}, \lab{f}}$, we have $\icn{x}{\calG}{F} = \cn{x}{\calG}{F}$.\label{item:ulf-forest} 
        \item If~$\calG$ is hereditary and~$F$ is a forest, we have $\icn{g}{\calG}{F} \leq 2\cn{g}{\calG}{F}$.\label{item:g-forest-G_her}
        \item\label{itm:forest_comp_closed} If~$\calG$ is component-closed and~$F$ is a star forest, we have $\icn{g}{\calG}{F} = \cn{g}{\calG}{F}$.\label{item:g-star-G_her}
    \end{enumerate}
\end{lemma}
\begin{proof}
    To prove \eqref{item:ulf-forest}, it suffices to observe that each connected subgraph~$G$ of a forest~$F$ is an induced subgraph of~$F$. 
    This holds in particular for each component of the image~$\varphi(G_i)$ of a guest~$G_i$ in a $\calG$-cover~$\varphi\colon G_1 \cupdot \dots \cupdot G_t \to F$ of~$F$.
    Thus, $\cn{x}{\calG}{F} = \icn{x}{\calG}{F}$ for every~$\lab{x} \in \set{\lab{u}, \lab{l}, \lab{f}}$.

    In order to prove~\eqref{item:g-forest-G_her}, we show that each subgraph~$F'$ of a forest~$F$ can be decomposed into two induced subgraphs of~$F$. 
    As $\calG$ is hereditary, it then follows that each guest graph of a $k$-global $\calG$-cover of~$F$ can be decomposed into two induced graphs,  yielding a $2k$-global-induced $\calG$-cover of~$F$.

    Now let~$F'$ be a subgraph of the forest~$F$. 
    It remains to show that~$F'$ can be decomposed into two induced subgraphs of~$F$.
    Consider the auxiliary graph~$H$ obtained by contracting each component of~$F'$ to a single vertex and joining two vertices if the corresponding components are joined by an edge in~$F$.
    As~$H$ is a forest and as such is bipartite, we can decompose the components of~$F'$ into two graphs~$F_1, F_2$ such that all components assigned to the same graph~$F_i$ have distance at least~$2$ in~$F$. 
    That is, $F_1$ and~$F_2$ are induced subgraphs of~$F$ covering all edges of~$F'$. 
    \begin{figure}
        \centering
        \begin{subfigure}[t]{.3\linewidth}
        \centering
        \includegraphics[width=0.8\linewidth,page=1]{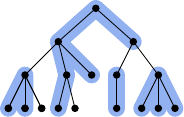}
        \caption{}
        \label{fig:forest}
        \end{subfigure}\hfil
        \begin{subfigure}[t]{.3\linewidth}
        \centering
        \includegraphics[width=0.8\linewidth,page=2]{figures/forests.pdf}
        \caption{}
        \label{fig:forest_auxiliary}
        \end{subfigure}\hfil
        \begin{subfigure}[t]{.3\linewidth}
        \centering
        \includegraphics[width=0.8\linewidth,page=3]{figures/forests.pdf}
        \caption{}
        \label{fig:forest_two_induced}
        \end{subfigure}
        \caption{(\subref{fig:forest}) A forest~$F$ ($\blacksquare$) and a (non-induced) subgraph~$F'$ (\textcolor{signalblue3}{$\blacksquare$}) of~$F$. (\subref{fig:forest_auxiliary}) A $2$-vertex-coloring of the auxiliary graph~$H$. (\subref{fig:forest_two_induced}) The corresponding decomposition of~$F'$ into two induced subgraphs~$F_1$ (\textcolor{signalblue4}{$\blacksquare$}) and~$F_2$ (\textcolor{signalblue2}{$\blacksquare$}).}
        \label{fig:decompose_forest_into_two_induced}
    \end{figure}

    To prove~\eqref{item:g-star-G_her}, observe that each subgraph of a star~$S$ is either an induced subgraph of~$S$ or contains an isolated vertex. 
    As $\calG$ is component-closed, we may assume that no guest of a $k$-global $\calG$-cover $\varphi\colon G_1 \cupdot \dots \cupdot G_t \to S$ contains isolated vertices. 
    That is, $\varphi$ is in particular a $k$-global-induced $\calG$-cover. 
    The same holds for star forests.
\end{proof}

In order to establish lower bounds on induced covering numbers, we argued in \cref{lem:sep_gul_H_bounded_tw_G_cc} that only small graphs of the guest class~$\calG$ appear as induced subgraphs of a host~$H$.
That is, high-degree vertices of~$H$ are hit by many guests in every injective $\calG$-cover.
For folded covering numbers it is substantially harder to establish such lower bounds:  if a cover~$\varphi \colon G_1 \cupdot \dots \cupdot G_t \to H$ is not injective, the graphs~$G_i$ need not be subgraphs of the host~$H$.
However, if~$G_i$ contains an odd cycle, so does~$\varphi(G_i)$. 

\begin{observation}
\label{obs:cover_odd_cycles}
    Let~$C_{\ell}$ be a cycle of odd size, $H$ be a graph and~$\varphi \colon C_{\ell} \to H$ be a graph homomorphism.
    Each of the following holds.
    \begin{enumerate}[(i)]
        \item $\varphi(C_{\ell})$ contains a cycle.
        \item If~$H = C_k$ for some~$k$, then $k \leq \ell$. 
    \end{enumerate}
\end{observation}

We now generalize the approach of \cref{lem:sep_gul_H_bounded_tw_G_cc} to folded-induced covering numbers.
That is, we create hosts which are structurally very similar to fans, and guests that behave similar to broken fans:
Let~$P$ be the path on vertices~$x_3, x_2, x_1, y_1, y_2, y_3$ (in order) and $P'$ the spanning subgraph where the edge~$x_1y_1$ is missing.
Now add a universal vertex~$f$ ($f'$) to $P$ ($P'$), i.e., a vertex adjacent to every vertex of~$P$ ($P'$).

For a graph~$G$ and a vertex~$v \in V(G)$, let~$\tilde{G}$ be the graph obtained from~$G$ by adding a copy of a cycle~$C$ and joining every vertex of~$C$ with an edge to~$v$. 
We say that~$\tilde{G}$ is obtained from~$G$ by \emph{joining} the cycle~$C$ to~$v$.
Joining
\begin{itemize}
    \item one copy of a cycle~$C_{2k+1}$ to the vertices~$x_2, y_3$ respectively,
    \item $k$ copies of cycles~$C_{2k+1}$ to the vertices~$x_3,y_2$ respectively,
    \item $k^2$ copies of cycles~$C_{2k+3}$ to the vertices~$x_1,y_1$ respectively,
    \item and attaching $k^k$~leaves to~$f$
\end{itemize}
yields a \emph{wheeled (broken) fan}~$W_{k}$ ($W_{k}')$; see \cref{fig:wheeled_fan} for an illustration.
\begin{figure}
    \centering
    \begin{subfigure}[t]{.5\linewidth}
    \centering
    \includegraphics[page=1]{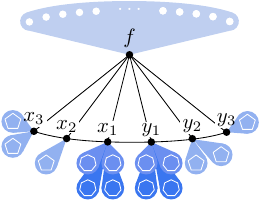}
    \caption{}
    \label{fig:wheeled_fan}
    \end{subfigure}\hfil
    \begin{subfigure}[t]{.5\linewidth}
    \centering
    \includegraphics[page=2]{figures/wheeled_fans.pdf}
    \caption{}
    \label{fig:wheeled_broken_fan}
    \end{subfigure}
    \caption{(\subref{fig:wheeled_fan}) A wheeled fan~$W_k$ and (\subref{fig:wheeled_broken_fan}) a wheeled broken fan~$W_k'$ for $k=2$.}
\end{figure}
We call the underlying graph~$P$ ($P'$) its \emph{(broken) path} and~$f$ ($f'$) the tip of~$W_k$ ($W_k'$).

\begin{proposition}
\label{lem:sep_gulf_H_tw_2_G_cc}
    There is a host class~$\calH$ with $\tw(\calH) \leq 3$ and a component-closed guest class~$\calG$ such that for every $\lab{x} \in \set{\lab{g}, \lab{u},\lab{l}, \lab{f}}$, the host class~$\calH$ is not $(\icn{x}{\calG}{},\cn{x}{\calG}{})$-bounded.
\end{proposition}
\begin{proof}
    Let the host class~$\calH$ consist of all wheeled fans, i.e. $\calH = \set{W_{k} \given k \geq 1}$.
    The guest class~$\calG$ consists of~$K_2$ and all wheeled broken fans~$W_{k}'$, i.e., $\calG = \set{W_{k}' \given k \geq 1} \cup \set{K_2}$.
    Observe that~$\calG$ is component-closed and $\tw(\calH) \leq 3$.
    
    In order to show that the host class~$\calH$ is not $(\icn{x}{\calG}{},\cn{x}{\calG}{})$-bounded for any $\lab{x} \in \set{\lab{g}, \lab{u}, \lab{l}, \lab{f}}$, it suffices to show the following: for every sufficiently large~$k$, we have
    \begin{enumerate}[(i)]
        \item\label{itm_f_rel:wheel_cn} $\cn{f}{\calG}{W_{k}} \leq \cn{l}{\calG}{W_{k}} \leq \cn{u}{\calG}{W_{k}} \leq \cn{g}{\calG}{W_{k}} \leq 2$
        \item\label{itm_f_rel:wheel_icn} $\icn{g}{\calG}{W_{k}} \geq \icn{u}{\calG}{W_{k}} \geq \icn{l}{\calG}{W_{k}} \geq \icn{f}{\calG}{W_{k}} \geq \frac{1}{3}k$.
    \end{enumerate}

    To prove \eqref{itm_f_rel:wheel_cn}, observe that $W_{k}'$ covers all but one edge (namely the edge~$x_1y_1$) of~$W_{k}$. 
    Thus, $\varphi\colon W_{k}' \cupdot K_2 \to W_{k}$ is a $2$-global $\calG$-cover certifying $\cn{g}{\calG}{W_{k}} \leq 2$.

    \smallskip

    To prove \eqref{itm_f_rel:wheel_icn}, consider an $\ell$-folded-induced $\calG$-cover $\varphi\colon G_1 \cupdot \dots \cupdot G_t \to W_{k}$.
    Here, the cover~$\varphi$ need not be injective.
    We need to show that $\ell \geq \frac{1}{3}k$ if~$k$ is sufficiently large.
    Let~$f$ denote the tip of~$W_{k}$ and let~$P$ be its underlying path on vertices~$x_3,x_2,x_1,y_1,y_2,y_3$ (in order).

    We show that for sufficiently large~$k$, we have
    $\ell \geq \frac{1}{3}k$,
    or each broken wheeled fan that corresponds to a cover graph~$G_i$ covers at most~$3k^{k-1}$ edges incident to the tip~$f$.
    As $\deg_{W_k}(f) \geq k^k$, it then follows that~$\ell \geq \frac{1}{3}k$, which concludes the proof of \eqref{itm_f_rel:wheel_icn}.
    
    Let~$W_{q}'$ be a broken wheeled fan that corresponds to a cover graph~$G_i$ and let~$f'$ be its tip and~$P'$ its underlying broken path on vertices~$x_3',x_2',x_1',y_1',y_2',y_3'$ (in order).

    Case 1. $q \geq k+1$. 
    As $\deg_{W_q'}(f') \geq k^{k+1}$ and $\deg_{W_{k}}(\varphi(f')) \leq 3k^k$ for sufficiently large~$k$, \cref{lem:degree}\,\eqref{itm:degree_folded} yields $\ell \geq \frac{1}{3}k$.

    Case 2. $q < k$. 
    As $\abs{V(W_q')} \leq 3(k-1)^{k-1}$ for sufficiently large~$k$, $W_q'$ covers at most $3k^{k-1}$ of the edges incident to~$f$.
   
    Case 3. $q = k$. 
    We prove that $\ell \geq \frac{1}{3}k$ or each of the following holds:
    \begin{enumerate}
        \item\label{itm:proof_bases} We have $\varphi(f') = f$.
        \item\label{itm:proof_paths} We have $\varphi(V(P')) \subseteq V(P)$.
        \item \label{itm:proof_xi_yi} For all $i \in \set{2,3}$, we have $\varphi(x_i'), \varphi(y_i') \in \set{x_3,x_2,y_2,y_3}$.
        \item \label{itm:proof_x1} We have $\varphi(x_1') = x_1$.
        \item \label{itm:proof_y1} We have $\varphi(y_1') = y_1$.
        \item \label{itm:proof_xy} There is no edge~$e \in E(W_k')$ such that $\varphi(e) = x_1y_1$.
    \end{enumerate}
    If each of the claims above holds, $\varphi(W_k')$ is not an induced subgraph of~$W_k$.
    Indeed, both $x_1$ and~$y_1$ lie in $\varphi(W_k')$, yet $x_1y_1 \notin E(\varphi(W_k'))$; a contradiction to the choice of~$\varphi$. 
    That is, $\ell \geq \frac{1}{3}k$.
    
    We first argue that $\varphi(f') = f$ or $\ell \geq \frac{1}{3}k$, i.e., \eqref{itm:proof_bases} holds.
    Suppose $f'$ is mapped to a vertex~$v$ other than~$f$.
    As $\deg_{W_q'}(f') \geq k^k$ and $\deg_{W_{k}}(v) \leq 3k^3$ for sufficiently large~$k$, \cref{lem:degree}\,\eqref{itm:degree_folded} yields $\ell \geq \frac{1}{3}k$.

    We may thus assume that $\varphi(f') = f$.
    All neighbors of~$f'$ are mapped to neighbors of~$f$.
    If a vertex~$z$ of the broken path~$P'$ is mapped to a leaf of~$f$, then all neighbors of~$z$ are mapped to~$f$. 
    Yet, the neighborhood of~$z$ contains two adjacent vertices (e.g. on a cycle joined to~$z$), and we therefore obtain a loop at~$f$; a contradiction.
    Thus, all vertices of the broken path~$P'$ are mapped to vertices of~$P$ and \eqref{itm:proof_paths} follows.
    
    To prove \eqref{itm:proof_xi_yi}, let~$v' \in \set{x_3',x_2',y_2',y_3'}$.
    By \eqref{itm:proof_paths}, it suffices to show that $\varphi(v') \notin \set{x_1, y_1}$. 
    Suppose $\varphi(v') = x_1$. 
    Note that a cycle~$C_{2k+1}$ attached to~$v'$ has to be mapped to a cycle~$C_{2k+3}$ attached to~$x_1$.
    Indeed, as $v'$ is universal for each of the copies $C_{2k+1}$, no vertex of $C_{2k+1}$ can be mapped to~$\varphi(v')$ and the only odd cycles in the neighborhood of~$x_1$ are cycles~$C_{2k+3}$.
    Now \cref{obs:cover_odd_cycles} yields a contradiction.
    That is~$\varphi(v') \neq x_1$. 
    Similarly, we obtain $\varphi(v') \neq y_1$ and \eqref{itm:proof_xi_yi} follows.

    Observe that $\varphi(x_1') \in \set{x_1,y_1}$.
    Indeed, by \eqref{itm:proof_paths}, we have $\varphi(x_1') \in V(P)$; and since $\deg_{W_{k}'}(x_1') \geq k^3$, yet $\deg_{W_k}(v) \leq 3k^2$ for all $v \in V(P) - \set{x_1,y_1}$ (and sufficiently large~$k$), \cref{lem:degree}\,\eqref{itm:degree_folded} yields $\ell \geq \frac{1}{3}k$ or $\varphi(x_1') \in \set{x_1,y_1}$ as desired.

    Now suppose $\varphi(x_1') = y_1$. 
    As \eqref{itm:proof_paths} holds and~$x_2'$ is a neighbor of~$x_1'$, $\varphi(x_2') \in \set{x_1,y_2}$ follows.
    By \eqref{itm:proof_xi_yi}, we obtain $\varphi(x_2') = y_2$. 
    A similar argument shows $\varphi(x_3') = y_3$.  
    Yet, $\deg_{W_k'}(x_3') \geq k^2$ and $\deg_{W_k}(y_3) \leq 3k$ for sufficiently large~$k$. 
    Thus, \cref{lem:degree}\,\eqref{itm:degree_folded} yields $\ell \geq \frac{1}{3}k$ for $k$ sufficiently large and \eqref{itm:proof_x1} follows. 
    A similar argument yields \eqref{itm:proof_y1}.

    It remains to prove \eqref{itm:proof_xy}.
    Suppose there is an edge~$u'v' \in E(W_k')$ with~$\varphi(u'v') = x_1y_1$. 
    Neither~$u'$ nor~$v'$ can correspond to a leaf of~$f'$.
    As each $x_i'$ ($y_i')$ is universal for each of its attached cycles and~$\varphi(N(x_i')) \subseteq N(\varphi(x_i))$, \eqref{itm:proof_xi_yi} shows that $u'v'$ needs to correspond to
    \begin{itemize}
    \item an edge of a cycle~$C' = C_{2k+3}$ attached to~$x_1'$ or~$y_1'$, 
    \item or an edge connecting $C'=C_{2k+3}$ to $x_1'$ or~$y_1'$.
    \end{itemize}
    We may assume that~$C'$ is attached to~$x_1'$. 
    As~$x_1'$ is universal in~$C'$, no vertex of~$C'$ is mapped to~$\varphi(x_1') = x_1$ (cf. \eqref{itm:proof_x1}), refuting the first option.
    Thus, $u'v'$ corresponds to an edge connecting~$C'$ to~$x_1'$ and~$\varphi(v') = y_1$. 
    Yet, all vertices of~$C'$ need to be mapped to neighbors of~$x_1$. 
    As there are only three neighbors of~$x_1$ which are not part of a cycle attached to~$x_1$, either one of these is hit at least~$\frac{1}{3}k$ times or some vertex of~$C'$ is mapped to~$x_1$.
    In the latter case, we obtain a loop in~$W_k$, a contradiction.
    Thus, $\ell \geq \frac{1}{3}k$ or \eqref{itm:proof_xy} holds.
\end{proof}

\section{Hereditary guests and sparse hosts}

When there is no restriction on the guest class~$\calG$, in general, the induced covering number~$\icn{x}{\calG}{}$ cannot be bounded in terms of~$\cn{x}{\calG}{}$ (cf. \cref{sec:general_guests_tw}).
In this section we thus turn to hereditary guest classes and consider sparse hosts.
Here, sparsity refers to host classes that are $M$-minor-free or of bounded maximum average degree.
We prove all results in \cref{tab:overview} which correspond to ($\icn{x}{\calG}{},\cn{x}{\calG}{}$)-boundedness in columns~3-5 and the separation result in column~8.

When the guest class~$\calG$ is hereditary, covering numbers are monotone with respect to taking induced subgraphs. 
That is, if~$H'$ is an induced subgraph of a graph~$H$, then $\cn{x}{\calG}{H'} \leq \cn{x}{\calG}{H}$.
\begin{lemma}[{Goetze, Stumpf, Ueckerdt \cite[Lemma~9]{goetze2025boundednessseparationgraphcovering}}]
    \label{lem:restrict-cover-to-subgraph}
    Let $\calG$ be any graph class and $\lab{x} \in \set{g,u,l,f}$.
    For every graph $H$ each of the following holds.
    \begin{enumerate}[(i)]
        \item If $\calG$ is monotone, then for every subgraph $H'$ of $H$ we have $\cn{x}{\calG}{H'} \leq \cn{x}{\calG}{H}$. \label{item:cover_subgraph_G_mon}
        \item If $\calG$ is hereditary, then for every induced subgraph $H'$ of $H$ we have $\cn{x}{\calG}{H'} \leq \cn{x}{\calG}{H}$.
        \label{item:cover_ind_subgraph_G_her}
        \item If $\calG$ is hereditary and $\lab{x} \neq \lab{g}$, then for every weak induced subgraph $H'$ of $H$ we have $\cn{x}{\calG}{H'} \leq \cn{x}{\calG}{H}$. \label{item:cover_weak_ind_subgraph_G_her}
    \end{enumerate}
\end{lemma}

In order to prove ($\icn{x}{\calG}{},\cn{x}{\calG}{}$)-boundedness, it thus suffices to decompose a graph~$H$ into smaller parts~$H_1', \dots, H_t'$ for which we can bound the induced covering number~$\icn{x}{\calG}{}$ in terms of the covering number~$\cn{x}{\calG}{}$. 
If the parts~$H_i'$ are weak induced subgraphs of~$H$, the union of the weak induced $\calG$-covers of the graphs~$H_i'$ forms a weak induced $\calG$-cover of~$H$.

\begin{lemma}
\label{lem:cover_with_small_parts}
    Let~$\calG$ be a hereditary graph class and~$\calH'$ a graph class such that there exists a function~$f \colon \N \to \N$ with $\icn{x}{\calG}{H'} \leq f(\cn{x}{\calG}{H'})$ for all~$H' \in \calH'$.
    Let~$\lab{x} \in \set{\lab{g},\lab{u},\lab{l},\lab{f}}$ and~$H$ be a graph that is the union of~$t$ graphs~$H_1', \dots, H_t'$ with~$H_i' \in \calH'$. 
    \begin{itemize}
        \item If~$\lab{x} \neq \lab{g}$ and each~$H_i'$ is a weak induced subgraph of~$H$, then $\icn{x}{\calG}{H} \leq t \cdot f(\cn{x}{\calG}{H})$. \label{itm:cover_with_weak_induced}
        \item If~$\lab{x} = \lab{g}$ and each~$H_i'$ is an induced subgraph of~$H$, then $\icn{g}{\calG}{H} \leq t \cdot f(\cn{g}{\calG}{H})$. \label{itm:cover_with_induced}
    \end{itemize}
\end{lemma}
\begin{proof}
    \cref{lem:restrict-cover-to-subgraph} yields $\icn{x}{\calG}{H_i'} \leq f(\cn{x}{\calG}{H_i'}) \leq f(\cn{x}{\calG}{H})$ for all $i$ as $\calG$ is hereditary.
    As each~$H_i'$ is a (weak) induced subgraph of~$H$, their union forms a (weak) induced cover of~$H$. 
    We thus obtain 
    \[
    \icn{x}{\calG}{H} \leq \sum_{i=1}^{t} \icn{x}{\calG}{H_i'} \leq \sum_{i=1}^{t} f(\cn{x}{\calG}{H}) \leq t \cdot f(\cn{x}{\calG}{H}).\qedhere
    \]
\end{proof}

Forests are a good choice for such small parts~$H_i'$.
Indeed, for every forest~$F$, we have~$\icn{x}{\calG}{F} \leq 2\cn{x}{\calG}{F}$ if~$\calG$ is hereditary (cf. \cref{lem:forests}) and every graph~$H$ of small maximum average degree can be decomposed into few weak induced forests:

\begin{lemma}[Axenovich, Dörr, Rollin, Ueckerdt {\cite[Thereom~7,8]{axenovich2019induced}}]
    \label{lem:mad_bounded_split_into_weak_induced_star_forests}
    Let~$H$ be a graph.
    \begin{enumerate}[(i)]
        \item If~$\mad(H) \leq d$, then $H$ is the union of at most $2d$~weak induced (star) forests.
        \item If~$\tw(H) \leq d$, then $H$ is the union of at most $\frac{1}{2}(d+1)^2$ induced forests.
    \end{enumerate}
\end{lemma}

\begin{corollary}
\label{prop:bounded_H_mad_tw_G_hereditary}
    Let $\calG$ be any hereditary graph class and $\lab{x} \in \set{\lab{g},\lab{u},\lab{l},\lab{f}}$.
    For every graph~$H$ each of the following holds.
     \begin{enumerate}[(i)]
         \item If~$\lab{x} \neq \lab{g}$, then $\icn{x}{\calG}{H} \leq 2\mad(H) \cdot \cn{x}{\calG}{H}$. \label{itm:mad}
         \item If~$\lab{x} = \lab{g}$, then $\icn{g}{\calG}{H} \leq (\tw(H)+1)^2 \cdot \cn{g}{\calG}{H}$. \label{itm:tw}
     \end{enumerate}
     In particular, for hereditary guest classes $\calG$, every host class $\calH$ of bounded maximum average degree is $(\icn{x}{\calG}{},\cn{x}{\calG}{})$-bounded for $\lab{x} \in \set{\lab{u},\lab{l},\lab{f}}$.
\end{corollary}
\begin{proof}
    By \cref{lem:mad_bounded_split_into_weak_induced_star_forests}, $H$ is the union of at most $2\mad(H)$ weak induced forests.
    Now \eqref{itm:mad} follows from \cref{lem:cover_with_small_parts} as we have $\icn{x}{\calG}{F}=\cn{x}{\calG}{F}$ for every forest~$F$ and $\lab{x} \neq \lab{g}$ by \cref{lem:forests}.
    
    Similarly, \cref{lem:mad_bounded_split_into_weak_induced_star_forests,lem:forests} show that~$H$ is the union of at most $\frac{1}{2}(\tw(H)+1)^2$ induced forests and $\icn{g}{\calG}{F}\leq 2\cn{g}{\calG}{F}$ for every forest~$F$.
    An application of \cref{lem:cover_with_small_parts} yields \eqref{itm:tw}.
\end{proof}

In fact, in some cases the maximum average degree~$\mad(H)$ of a graph~$H$ can be bounded in terms of the covering number~$\cn{x}{\calG}{}$ for a guest class~$\calG$ and~$\mad(\calG)$ (cf. \cref{lem:mad_G_bounded_bounded_folded_cov_for_induced_subgraphs_then_mad_H_bounded}).
Thereby, we obtain a similar result to \cref{prop:bounded_H_mad_tw_G_hereditary} for hereditary guest classes of bounded maximum average degree (cf. \cref{prop:bounded_H_any_G_hereditary_mad}).

\begin{lemma}[Goetze, Stumpf, Ueckerdt {\cite[Lemma~12]{goetze2025boundednessseparationgraphcovering}}]
\label{lem:mad_G_bounded_bounded_folded_cov_for_induced_subgraphs_then_mad_H_bounded}
    Let~$\calG$ be a graph class with $\mad(\calG) \leq d$ and let~$H$ be a graph.
    If there exists a constant~$s$ such that for every induced subgraph~$H'$ of~$H$ we have $\cn{f}{\calG}{H'} \leq s$, then $\mad(H) \leq sd$.
\end{lemma}

\begin{corollary}
\label{prop:bounded_H_any_G_hereditary_mad}
    If~$\calG$ is a hereditary guest class with $\mad(\calG) \leq d$, then we have for every graph~$H$ and every $\lab{x} \in \set{\lab{u},\lab{l},\lab{f}}$
    \begin{align*}
        \icn{x}{\calG}{H} \leq 2d \cdot \cn{x}{\calG}{H}^2.
    \end{align*}
     In particular, for such guest classes $\calG$, every host class $\calH$ is $(\icn{x}{\calG}{},\cn{x}{\calG}{})$-bounded.
\end{corollary}
\begin{proof}
Consider a graph~$H$.
As~$\calG$ is hereditary, we have $\cn{f}{\calG}{H'} \leq \cn{x}{\calG}{H'} \leq \cn{x}{\calG}{H}$ for every induced subgraph~$H'$ of~$H$ (cf. \cref{lem:restrict-cover-to-subgraph}).
Thus,  $\mad(H) \leq \cn{x}{\calG}{H} \cdot d$ by \cref{lem:mad_G_bounded_bounded_folded_cov_for_induced_subgraphs_then_mad_H_bounded}. 
Now \cref{prop:bounded_H_mad_tw_G_hereditary} yields $\icn{x}{\calG}{H} \leq 2\mad(H) \cdot \cn{x}{\calG}{H} \leq 2d \cdot (\cn{x}{\calG}{H})^2$.
\end{proof}

In \cref{prop:bounded_H_mad_tw_G_hereditary,prop:bounded_H_any_G_hereditary_mad} we obtained $(\icn{x}{\calG}{},\cn{x}{\calG}{})$-boundedness for hereditary guest classes~$\calG$ when the host class~$\calH$ or the guest class~$\calG$ has bounded maximum average degree in all settings but for~$\lab{x} = \lab{g}$.
Yet, similar results do not hold for global covering numbers.
\begin{proposition}[{Axenovich, Dörr, Rollin, Ueckerdt~\cite[Theorem 4 (ii) (a)]{axenovich2018kstronginducedarboricity}\protect\footnotemark}]
\label{lem:sep_g_H_mad_G_monotone}
\footnotetext{While they prove a slightly different statement, their proof also yields \cref{lem:sep_g_H_mad_G_monotone}. We restate their proof.}
There is a host class~$\calH$ with~$\mad(\calH) \leq 2$ and a monotone guest class~$\calG$ with $\mad(\calG) \leq 1$ such that~$\calH$ is not $(\icn{g}{\calG}{},\cn{g}{\calG}{})$-bounded.
\end{proposition}
\begin{proof}
    Let~$K_n'$ denote the \emph{$1$-subdivision} of~$K_n$, that is the graph obtained from $K_n$ by replacing each edge with a path on three vertices.
    We call the vertices introduced by the replacement \emph{subdivision vertices}. 
    The guest class~$\calG$ consists of all star forests,
    the host class~$\calH$ of all $1$-subdivisions of complete graphs, i.e. $\calH = \set{K_n' \given n \in \N}$. 
    Note that~$\calG$ is monotone and $\mad(\calG) \leq 1$ and $\mad(\calH) \leq 2$.
    
    It remains to show that
    \begin{enumerate}[(i)]
        \item\label{itm:Kn_subdiv_cn} $\cn{g}{\calG}{K_n'} \leq 2$ for every $n \in \N$. 
        \item\label{itm:Kn_subdiv_icn} For every $n \in \N$, there exists $N \geq n$ such that $\icn{g}{\calG}{K_N'} \geq n$.
    \end{enumerate}
    To prove \eqref{itm:Kn_subdiv_cn}, consider any (not necessarily proper) $2$-edge-coloring of~$K_n'$ where no two edges incident to the same subdivision vertex share a color. 
    As each of the two color classes corresponds to a star forest, and their union covers all edges of~$K_n'$, we obtain $\cn{g}{\calG}{K_n'} \leq 2$.

    In order to show \eqref{itm:Kn_subdiv_icn}, we consider the minimum number~$N$ of vertices such that every $(n+1)^2$-edge-coloring of the complete graph~$K_N$ contains a monochromatic triangle (the number~$N$ is finite by Ramsey's theorem \cite{wan1997upperBoundsRamsey}).
    
    Consider an induced cover~$\varphi' \colon F_1 \cupdot \dots \cupdot F_t \to K_N'$ of~$K_N'$ with $t = \icn{g}{\calG}{K_N'}$~star forests.
    It now suffices to construct an edge-coloring~$\varphi\colon E(K_N) \to [(t+1)^2]$ of the complete graph~$K_N$ with at most $(t+1)^2$~colors which contains no monochromatic triangle. 
    Indeed, as every such edge-coloring uses more than $(n+1)^2$~colors, we obtain $(\icn{g}{\calG}{K_N'}+1)^2 = (t+1)^2 > (n+1)^2$ and \eqref{itm:Kn_subdiv_icn} follows.

    To construct the edge-coloring~$\varphi$, consider for every edge~$e \in E(K_N)$ the two edges~$e_1, e_2 \in E(K_N')$ obtained through subdivision of~$e$.
    If there exists a forest~$F_i$ such that $e_1, e_2 \in E(F_i)$, set $\varphi(e) = i$. 
    If there is no such forest, there exist two forests~$F_i,F_j$ with $i \neq j$ such that~$e_1 \in E(F_i), e_2 \in E(F_j)$ (choose any such pair). 
    We now set $\varphi(e) = \set{i,j}$. 
    Clearly, $\varphi$ uses at most $t + \binom{t}{2} = \binom{t+1}{2}$~colors. 
    As each~$F_i$ is a forest, $\varphi$ contains no monochromatic triangle in any color~$i$.
    Neither does $\varphi$ contain a monochromatic triangle in a color~$\set{i,j}$ as each of the forests~$F_i, F_j$ is induced in~$K_N'$, see \cref{fig:kn_1-subdivision}.
    \begin{figure}
        \centering
        \begin{subfigure}[t]{.4\linewidth}
            \centering
            \includegraphics[page=3]{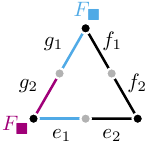}
            \caption{}
            \label{fig:kn_1-subdision:kn_1-forests_partial}
        \end{subfigure}
        \begin{subfigure}[t]{.4\linewidth}
            \centering
            \includegraphics[page=4]{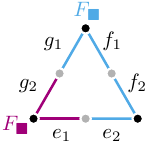}
            \caption{}
            \label{fig:kn_1-subdision:kn_1-forests}
        \end{subfigure}
        
        \caption{If there is a monochromatic triangle in~$K_N$ on edges~$e,f,g$ in color $\set{\textcolor{KITcyanblue}{\blacksquare},\textcolor{KITlilac}{\blacksquare}}$ each subdivision vertex (represented by $\textcolor{KITgray30}{\bullet}$) on the corresponding cycle~$C_6$ in~$K_N'$ is incident to edges of both colors. Thus, each original vertex (represented by $\bullet$) of~$C_6$ is only incident to one of the two colors (otherwise one of the forests~$F_{\textcolor{KITcyanblue}{\blacksquare}}$, $F_{\textcolor{KITlilac}{\blacksquare}}$ would not be induced, see (\subref{fig:kn_1-subdision:kn_1-forests_partial})). Yet, this results in two edges $f_1,f_2$ being of the same color, a contradiction to~$e,f,g$ being colored in $\set{\textcolor{KITcyanblue}{\blacksquare},\textcolor{KITlilac}{\blacksquare}}$ (\subref{fig:kn_1-subdision:kn_1-forests}).}
        \label{fig:kn_1-subdivision}
    \end{figure}
\end{proof}

The global-induced covering number~$\icn{g}{\calG}{H}$ cannot be bounded in terms of~$\mad(\calG)$, $\mad(H)$ and~$\cn{g}{\calG}{H}$. 
We thus now turn to hosts which are even more restricted: host classes~$\calH$ which are $M$-minor-free for some graph~$M$.
For such host classes we obtain~$(\icn{g}{\calG}{},\cn{g}{\calG}{})$-boundedness.
\begin{theorem}
\label{lem:bounded_g_H_minor_G_hereditary}
    If~$\calG$ is hereditary, then we have for every $K_k$-minor free graph~$H$
    \[
    \icn{g}{\calG}{H} \leq \frac{25(k-1)^4}{8}\cdot \cn{g}{\calG}{H}.
    \]
    In particular, for every graph~$M$, every $M$-minor free host class~$\calH$ is $(\icn{g}{\calG}{},\cn{g}{\calG}{})$-bounded.
\end{theorem}

To prove above theorem, we cover the edges of a graph~$H$ with few induced star forests.
In fact, this is possible if~$H$ has small \emph{star chromatic number}. 
The star chromatic number~$\starchrom(H)$ of a graph~$H$ is the smallest number~$c$ of colors such that there exists a proper $c$-vertex-coloring of~$H$ where every pair of two color classes induces a star forest.
Using a relation between the star chromatic number and what is known as the \emph{strong $2$-coloring number} the authors of \cite[p.\,131]{van2017generalisedcolouringnumbers} obtain an upper bound on~$\starchrom(\calH)$ for every $K_k$-minor-free graph~$H$.
\begin{lemma}[{van den Heuvel, Ossana de Mendez, Quiroz, Rabinovich, Siebertz \cite[Cor\,1.3]{van2017generalisedcolouringnumbers}\protect\footnotemark}]
\footnotetext{In fact, the authors give an upper bound on the \emph{strong $2$-coloring number} \cite[Cor.\,1.3]{van2017generalisedcolouringnumbers} which corresponds to the star chromatic number \cite[p.\,188]{jiang2023chitochibounded}.}
\label{lem:star_chrom_for_minor-free}
    For every graph $H$ that excludes the complete graph $K_k$ as a minor, we have
    $\starchrom(H) \leq \frac{5}{2}\cdot (k-1)^2$.
\end{lemma}

\begin{proof}[Proof of \cref{lem:bounded_g_H_minor_G_hereditary}]
    As~$H$ is~$K_k$-minor-free, there exists by \cref{lem:star_chrom_for_minor-free} a proper~$s$-vertex-coloring~$\Phi\colon V(H) \to [s]$ of~$H$ with $s \leq \frac{5}{2}(k-1)^2$ colors such that the union of any two color classes induces a star forest in~$H$.
    For~$(i,j) \in I \coloneqq \set{(i,j) \in [s]^2 \given i < j}$, let~$S_{i,j}$ denote the star forest of~$H$ induced by the the two color classes $\Phi^{-1}(i)$ and~$\Phi^{-1}(j)$.
    The graphs~$S_{i,j}$ form an induced cover of~$H$ with $\frac{1}{2}s^2$~star forests.
    As $\calG$ is hereditary (and thus in particular component-closed), \cref{lem:forests}\,\eqref{itm:forest_comp_closed} yields $\icn{g}{\calG}{S_{i,j}} = \cn{g}{\calG}{S_{i,j}}$.
    That is, the condition of \cref{lem:cover_with_small_parts} is met, and we obtain 
    \[\icn{g}{\calG}{H} \leq \frac{1}{2}s^2\cdot \cn{g}{\calG}{H} \leq \frac{25}{8}(k-1)^4\cn{g}{\calG}{H}.\qedhere\]
\end{proof}

\section{Dense Hosts and Guests}
In the previous section we showed that if the host class~$\calH$ or the guest class~$\calG$ is sparse, and~$\calG$ is hereditary, then we obtain $(\icn{x}{\calG}{},\cn{x}{\calG}{})$-boundedness.
Yet, this cannot be extended to dense guest classes.
\begin{lemma}
\label{lem:sep_gulf_H_any_G_hereditary}
    There is a host class~$\calH$ and a hereditary guest class~$\calG$ such that for every $\lab{x} \in \set{\lab{g}, \lab{u},\lab{l},\lab{f}}$, the host class~$\calH$ is not $(\icn{x}{\calG}{},\cn{x}{\calG}{})$-bounded.
\end{lemma}
\begin{proof} 
    The guest class~$\calG = \overline{\set{K_n, K_{1,n} \given n \in \N}}$ consists of all unions of complete graphs and stars. 
    In particular, $\calG$ is hereditary.
    To define the host class~$\calH$, consider for every~$n \in \N$ the graph~$H_n$ obtained from $n$~disjoint copies of~$K_n$ by adding a universal vertex~$h$. We call~$h$ the center of~$H_n$, see \cref{fig:dense_hosts} for an example.
    We set~$\calH = \set{H_n \given n \in \N}$.
    \begin{figure}
        \centering
        \begin{subfigure}[t]{.4\linewidth}
            \centering
            \includegraphics[page=1]{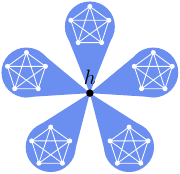}
            \caption{}
            \label{fig:dense_hosts}
        \end{subfigure}
        \begin{subfigure}[t]{.4\linewidth}
            \centering
            \includegraphics[page=2]{figures/sep_gulf_H_any_G_hereditary.pdf}
            \caption{}
            \label{fig:dense_hosts_induced_cover}
        \end{subfigure}
        
        \caption{(\subref{fig:dense_hosts}) A host~$H_n$ (for $n=5$) consisting of a universal vertex~$h$ joined to $n$~disjoint copies of the complete graph~$K_n$. (\subref{fig:dense_hosts_induced_cover}) If an induced subgraph~$H'$ of~$H_n$ contains the vertex~$h$, two vertices~$x,x'$ in one copy of~$K_n$ and a vertex~$y$ in a different copy, then~$H'$ is neither a star nor a complete graph.}
        \label{fig:sep_gulf_H_any_G_hereditary}
    \end{figure}

    In order to show that $\calH$ is not $(\icn{x}{\calG}{},\cn{x}{\calG}{})$-bounded for any~$\lab{x} \in \set{\lab{g}, \lab{u}, \lab{l}, \lab{f}}$, it suffices to prove the following: 
    for every~$n \geq 2$, we have
    \begin{enumerate}[(i)]
        \item\label{itm:H_n_cn} $\cn{f}{\calG}{H_n} \leq \cn{l}{\calG}{H_n} \leq \cn{u}{\calG}{H_n} \leq \cn{g}{\calG}{H_n} \leq 2$
        \item\label{itm:H_n_icn} $\icn{g}{\calG}{H_n} \geq \icn{u}{\calG}{H_n} \geq \icn{l}{\calG}{H_n} \geq \icn{f}{\calG}{H_n} \geq n$.
    \end{enumerate}

    To prove \eqref{itm:H_n_cn}, observe that the edges of~$H_n$ can be covered with the graph~$n \cdot K_n$ (the disjoint union of $n$~copies of~$K_n$) and the star~$K_{1,d}$ where $d = \deg_{H_n}(h) = n^2$ corresponds to the degree of the center~$h$ of~$H_n$.
    That is $\varphi\colon (n \cdot K_n) \cupdot K_{1,d} \to H_n$ is a $2$-global $\calG$-cover certifying $\cn{g}{\calG}{H_n} \leq 2$.

    To prove \eqref{itm:H_n_icn}, we need to show for $n \geq 2$ and a (not necessarily injective) $f$-folded-induced $\calG$-cover $\varphi\colon G_1 \cupdot \dots \cupdot G_t \to H_n$ that $f \geq n$. 
    We prove that each component~$C$ of a cover graph~$G_i$ that hits the center~$h$ of~$H_n$ covers at most~$n$ of the $n^2$~edges incident to~$h$.
    It then follows that~$h$ is hit at least~$n$ times and thus in particular $f \geq n$.
    As this holds for every weak induced $\calG$-cover of~$H_n$, we obtain $\icn{f}{\calG}{H_n} \geq n$.

    Let~$C$ be a component of a cover graph~$G_i$ that hits the center~$h$ of~$H_n$.
    It remains to show that~$C$ covers at most~$n$ edges incident to~$h$.
    If~$C$ covers more than $n$~edges incident to~$h$, then~$C$ covers at least two distinct edges~$xh,x'h$ connecting~$h$ to a copy~$X$ of~$K_n$, and at least one edge~$yh$ connecting~$h$ to a different copy~$Y$ of~$K_n$, see \cref{fig:dense_hosts_induced_cover} for an illustration. 
    As~$\varphi(C)$ is an induced subgraph of~$H_n$, we have $xx' \in E(\varphi(C))$. 
    Yet, $xy,x'y \notin E(H_n)$.  
    That is,~$\varphi(C)$ is neither a star nor a complete graph, a contradiction to~$C \in \calG$ as $\varphi(C)$ is a star if~$C$ is a star (possibly with loops at the center), and~$\varphi(C)$ is a complete graph (possibly with loops) if~$C$ is a complete graph.
\end{proof}

\section{Discussion}

In this paper we provide boundedness results, which yield quadratic upper bounds on induced covering numbers in terms of the corresponding (non-induced) covering numbers.
While these results are not exhaustive, they nonetheless generalize previous observations where induced and non-induced graph parameters have been compared (for example for arboricity \cite{axenovich2019induced}).

In \cref{discussion:boundedness_separation}, we discuss how an approach towards proving boundedness in several of the settings we left open (cells marked in $\textcolor{openColorCond}\blacksquare$ in \cref{tab:overview}) could work, reducing all these cases to proving one statement for bipartite hosts and guests (cf. \cref{conjecture:bounded_u_H_bip_G_hereditary_bip}). 
Further, we remark upon the similar behavior of union-, local- and folded-induced covering numbers with respect to boundedness and separation. 
In \cref{discussion:algorithms_complexity}, we consider open questions regarding the algorithmic complexity of induced covering number problems.

\subsection{Boundedness and Separation}
\label{discussion:boundedness_separation}

If we do no take the property of bounded chromatic number into account, only for one setting, namely when the guest class~$\calG$ is monotone, and there is no restriction on~$\calH$ we do not know whether each such~$\calH$ is $(\icn{x}{\calG}{},\icn{x}{\calG}{})$-bounded, see \cref{tab:overview}.
Yet, many cases (six in total) where~$\calH$ or~$\calG$ have bounded chromatic number remain open. 
In four of these (corresponding to the cells marked in $\textcolor{openColorCond}\blacksquare$ in \cref{tab:overview}) we obtain $(\icn{x}{\calG}{},\cn{x}{\calG}{})$-boundedness if the following hypothesis holds:
\begin{hypothesis}
    \label{conjecture:bounded_u_H_bip_G_hereditary_bip}
        If~$\calG$ is a hereditary graph class of bipartite graphs, then there exists a function~$f$, such that for every bipartite graph~$H$, we have 
        \[
        \icn{u}{\calG}{H} \leq f(\cn{u}{\calG}{H}).
        \]
        In other words, if the edges of~$H$ can be covered with $k$~graphs from the union-closure~$\overline{\calG}$ of~$\calG$, then~$H$ can be covered with up to~$f(k)$ graphs from~$\overline{\calG}$ whose components are induced subgraphs of~$H$.
\end{hypothesis}

Indeed, under the above hypothesis, we obtain the following theorem:
\begin{theorem}
    \label{cond_thm:bounded_c_gulf_H_chi_G_hereditary}
    Assuming \cref{conjecture:bounded_u_H_bip_G_hereditary_bip}, there exists for every hereditary graph class~$\calG$ a function~$f$ such that we have for every graph~$H$
    \begin{itemize}
        \item $\icn{u}{\calG}{H} \leq \chi(H)^2 \cdot f(\cn{u}{\calG}{H})$,
        \item and for every $\lab{x} \in \set{\lab{u}, \lab{l}, \lab{f}}\colon \quad \icn{x}{\calG}{H}\leq \chi(H)^2 \cdot f(\chi(H)^2 \cdot \cn{x}{\calG}{H}^2).$
    \end{itemize}
    In particular, for hereditary guest class~$\calG$, every host class~$\calH$ of bounded chromatic number is $(\icn{x}{\calG}{},\cn{x}{\calG}{})$-bounded for $\lab{x} \in \set{\lab{u}, \lab{l}, \lab{f}}$.
\end{theorem}
\begin{proof}
    We show the theorem with the following three results:
    \begin{enumerate}[(i)]
        \item\label{lem:cover_with_induced_bipartite_graphs} (folklore, see for example \cite[Lemma~15]{goetze2025boundednessseparationgraphcovering} for a proof) If $H$ is a graph with $\chi(H) \leq k$, then $H$ is the union of at most $\binom{k}{2}$ bipartite induced subgraphs.
        \item\label{thm:bounded_c_gulf_H_chi_G_hereditary} (Goetze, Stumpf and Ueckerdt \cite[Theorem~16]{goetze2025boundednessseparationgraphcovering}) If~$\calG$ is hereditary, then we have for every graph~$H$
        \[
            \cn{u}{\calG}{H}\leq \chi(H)^2 \cdot \cn{f}{\calG}{H}^2.
        \] 
        \item\label{claim:bounded_u_H_bip_G_hereditary}
        If~$\calG$ is hereditary, then there exists a (monotone) function~$f$, such that for every bipartite graph~$H$, we have $\icn{u}{\calG}{H} \leq f(\cn{u}{\calG}{H})$.
    \end{enumerate}
    We first show how \cref{cond_thm:bounded_c_gulf_H_chi_G_hereditary} follows from \eqref{lem:cover_with_induced_bipartite_graphs}-\eqref{claim:bounded_u_H_bip_G_hereditary}.
    Consider a graph~$H$ with~$\chi(H) \leq k$.
    By \eqref{lem:cover_with_induced_bipartite_graphs}, there exist up to~$k^2$ bipartite induced subgraphs $B_1, \dots, B_t \subseteq H$ such that $\varphi\colon B_1 \cupdot \dots \cupdot B_t \to H$ is a $k^2$-global-induced cover of~$H$. 
    By \eqref{claim:bounded_u_H_bip_G_hereditary}, each of these graphs~$B_1, \dots, B_t$ has a weak-induced $\overline{\calG}$-cover with at most $f(\cn{u}{\calG}{B_i}) \leq f(\cn{u}{\calG}{H})$ guests (as $\calG$ is hereditary, cf. \cref{lem:restrict-cover-to-subgraph}).
    These guests yield a weak-induced $k^2 \cdot f(\cn{u}{\calG}{H})$-global $\overline{\calG}$-cover of~$H$. 
    That is, $\icn{u}{\calG}{H} \leq k^2 \cdot f(\cn{u}{\calG}{H})$,
    which shows the first claim of \cref{cond_thm:bounded_c_gulf_H_chi_G_hereditary}.
    In particular, we obtain 
    \begin{align*}
    \label{ineq:c_u_bounded_in_chi_c_f}
    \tag{$\star$}
    \cn{u}{\calG}{H} \leq \icn{u}{\calG}{H} \leq k^2 \cdot f(\cn{u}{\calG}{H}) \leq k^2 \cdot f(k^2 \cn{f}{\calG}{H}^2).
    \end{align*}
    where we used \eqref{thm:bounded_c_gulf_H_chi_G_hereditary} in the last step.
    The second claim of \cref{cond_thm:bounded_c_gulf_H_chi_G_hereditary}  now follows from the following calculation: for every~$\lab{x} \in \set{\lab{u}, \lab{l}, \lab{f}}$ we have by \eqref{ineq:c_u_bounded_in_chi_c_f}
    \[
    \icn{x}{\calG}{H} \leq \icn{u}{\calG}{H} \leq k^2 \cdot f(k^2 \cdot \cn{f}{\calG}{H}^2) \leq k^2 \cdot f(k^2 \cdot \cn{x}{\calG}{H}^2). 
    \]

    It only remains to prove \eqref{claim:bounded_u_H_bip_G_hereditary}.
    Let~$H$ be a bipartite graph and consider an injective $\overline{\calG}$-cover $\varphi\colon G_1 \cupdot \dots \cupdot G_t \to H$.
    Each~$G_i$ is a subgraph of~$H$ and as such bipartite. That is $\varphi$ is a $\calG'$-cover of~$H$ where $\calG' = \set{G \in \calG \given \text{$G$ is bipartite}}$. 
    In particular, we have $\cn{u}{\calG}{H} = \cn{u}{\calG'}{H}$. 
    Assuming \cref{conjecture:bounded_u_H_bip_G_hereditary_bip}, there exists a function~$f$ such that $\icn{u}{\calG}{H} = \icn{u}{\calG'}{H} \leq f(\cn{u}{\calG'}{H}) = f(\cn{u}{\calG}{H})$ for every bipartite graph~$H$.
    In fact, we may assume that~$f$ is monotone.
    Indeed, otherwise we may consider the monotone function~$f'\colon \N \to \N$ with $f'(n) = \max\set{f(n') \given n' \leq n}$ for every~$n \in \N$, which shows \eqref{claim:bounded_u_H_bip_G_hereditary}.
\end{proof}

In fact, \cref{cond_thm:bounded_c_gulf_H_chi_G_hereditary} would also settle two more open cases for the induced union-covering number, as the chromatic number of a graph~$H$ can be bounded in terms of the union covering number~$\cn{u}{\calG}{H}$ and~$\chi(\calG)$ for every graph class~$\calG$ (cf. \cref{lem:chi_of_host_bounded_in_chi_of_guest_and_union}).
\begin{corollary}
    Assuming \cref{conjecture:bounded_u_H_bip_G_hereditary_bip}, there exists for every hereditary graph class~$\calG$ with~$\chi(\calG) \leq k$ a function~$f$ such that we have for every graph~$H$
    \[
        \icn{u}{\calG}{H}\leq k^{2\cn{u}{\calG}{H}} \cdot f(\cn{u}{\calG}{H}).
    \] 
    In particular, for every hereditary guest class $\calG$ of bounded chromatic number, every host class $\calH$ is $(\icn{u}{\calG}{},\cn{u}{\calG}{})$-bounded.
\end{corollary}

\begin{lemma}
\label{lem:chi_of_host_bounded_in_chi_of_guest_and_union}
    For every guest class~$\calG$ with $\chi(\calG) \leq k$ and every graph~$H$, we have
    \[
        \chi(H) \leq k^{\cn{u}{\calG}{H}}.
    \]
\end{lemma}
\begin{proof}
    Consider an (injective) $t$-global $\overline{\calG}$-cover $\varphi\colon G_1 \cupdot \dots \cupdot G_t \to H$. 
    Let~$\Phi_i\colon G_i \to [k]$ be a proper vertex-coloring of~$G_i$ for every~$i$.  
    Now consider the vertex-coloring~$\Psi$ of~$H$ where every vertex~$v \in V(H)$ is colored with a $t$-tuple~$(\Psi_1(v), \dots, \Psi_t(v))$ such that for every~$x \in V(G_i)$ with $\varphi(x) = v$, we have $\Psi_i(v) = \Phi_i(x)$.
    As $\varphi$ is injective, there is for every vertex~$v \in V(H)$ and~$i \in [t]$ at most one~$x \in V(G_i)$ with~$\varphi(x) = v$.
    Thus, such a coloring~$\Psi$ exists.

    It remains to show that the coloring~$\Psi$ is proper. 
    For every edge~$uv \in E(H)$ there exists an edge $xy \in E(G_i)$ for some~$i$ such that $\varphi(x) = u$ and~$\varphi(y) = v$.
    Since~$\Phi_i$ is proper, we have $\Psi_i(u) = \Phi_i(x) \neq \Phi_i(y) = \Psi_i(v)$.
    Thus, $\Psi$ is proper.
\end{proof}

\begin{remark}
    While the chromatic number~$\chi(H)$ of a graph~$H$ can be bounded in terms of $\chi(\calG)$ and $\cn{u}{\calG}{H}$ (cf. \cref{lem:chi_of_host_bounded_in_chi_of_guest_and_union}), it cannot be bounded in terms of $\chi(\calG)$ and $\cn{x}{\calG}{H}$ for $\lab{x} \in \set{\lab{l}, \lab{f}}$. 
    Indeed, shift graphs~$\shift(D)$ of directed graphs $D$ provide an example as $\cn{f}{\calB}{\shift(D)} \leq \cn{l}{\calB}{\shift(D)} \leq 2$ \cite[Proposition~21]{goetze2025boundednessseparationgraphcovering} where $\calB$ denotes the class of all bipartite graphs.
    Yet, $\chi(\shift(D)) \geq \chi(D)$ \cite[Problem 9.26 (a)]{lovasz1993combinatorial}.
\end{remark}

Recall that there is a hierarchy on the four induced covering numbers, that is for every graph class~$\calG$ and graph~$H$, we have
\[
    \icn{f}{\calG}{H} \leq \icn{l}{\calG}{H} \leq \icn{u}{\calG}{H} \leq \icn{g}{\calG}{H}.
\]
The open cases discussed above apart, one might also wonder (as has been done for covering numbers \cite{goetze2025boundednessseparationgraphcovering}) when there is $(\icn{g}{\calG}{},\icn{u}{\calG}{})$-, $(\icn{u}{\calG}{},\icn{l}{\calG}{})$- or $(\icn{l}{\calG}{},\icn{f}{\calG}{})$-boundedness.
In fact, in the case of covering numbers, boundedness and separation behave very similar for union, local and folded covering numbers \cite[p.\,21]{goetze2025boundednessseparationgraphcovering}.
We observe the same pattern with respect to $(\icn{x}{\calG}{},\cn{x}{\calG}{})$-boundedness (cf. \cref{tab:overview}). 
Indeed, while there are graph classes~$\calG$ and~$\calH$ such that~$\calH$ is not $(\icn{x}{\calG}{},\cn{x}{\calG}{})$-bounded for~$\lab{x} \in \set{\lab{u},\lab{l}}$ but $(\icn{f}{\calG}{},\cn{f}{\calG}{})$-bounded (cf. \cref{lem:sep_gul_H_bounded_tw_G_cc}) it seems that whenever there is separation or boundedness for some~$\lab{x} \in \set{\lab{u},\lab{l},\lab{f}}$ for every~$\calG$ and $\calH$ satisfying specific density properties, we also obtain the same result for all other $\lab{x}' \in \set{\lab{u},\lab{l},\lab{f}}$.
Is there a deeper relationship to be uncovered?

\subsection{Algorithms and Complexity}
\label{discussion:algorithms_complexity}
It would be of interest to study the computational complexity of determining induced covering numbers.
Recently, Lee, Liu and Tsai provided a result \cite{lee2026determining} which shows for many graph classes~$\calG$ that it is $\NP$-hard to decide whether $\cn{u}{\calG}{H} \leq k$ for a given graph~$H$ and integer~$k$.
Can this result be extended to the induced setting?

Some covering numbers, in particular for local variants, can be determined in polynomial time, this is for example the case for local star arboricity (covers with star forests) \cite[Theorem~25]{Knauer2016_3w3c1g}. 
That is, whenever we have $(\icn{l}{\calG}{},\cn{l}{\calG}{})$-boundedness for the underlying guest class~$\calG$ and the host class~$\calH$, we obtain an approximation algorithm for the corresponding induced covering number, see \cref{sec:introduction} for examples.
Can the algorithms for covering numbers be adapted to determine induced covering numbers? 
Here, a meta-theorem which settles the question for any guest class~$\calG$ satisfying some restrictions (as is the case for global covering numbers in \cite{lee2026determining}) would be of great interest.
It would also be interesting to study cases of guest classes~$\calG$ where the (non-induced) covering number can be determined in polynomial time, while this is $\NP$-hard in the induced setting.

\bibliographystyle{plainurl}
\bibliography{references}
\end{document}